\documentclass[12pt]{article}
\usepackage{amsmath,amssymb,amsthm,amscd}
\title{Computation of the canonical form
for the matrices of chains and cycles
of linear mappings\footnotetext{This is the author's version of a work that was published in Linear Algebra Appl. 376 (2004) 235--263.}}
\author{Vladimir V. Sergeichuk%
\thanks{The research was done while the
author was visiting the University of
Utah supported by NSF grant
DMS-0070503.}\\ Institute of
Mathematics\\ Tereshchenkivska 3, Kiev,
Ukraine\\sergeich@imath.kiev.ua}

\date{}

\begin{document}
 \maketitle
\begin{abstract}
Paul Van Dooren \cite{doo} constructed
an algorithm for the computation of all
irregular summands in Kronecker's
canonical form of a matrix pencil. The
algorithm is numerically stable since
it uses only unitary transformations.

We construct a unitary algorithm for
computation of the canonical form of
the matrices of a chain of linear
mappings
\[
V_1 \,\frac{}{\qquad}\,
V_2\,\frac{}{\qquad}\, \cdots
\,\frac{}{\qquad}\, V_t
\]
and extend Paul Van Dooren's algorithm
to the matrices of a cycle of linear
mappings\\[-2mm]
\[
 \unitlength 0.8mm
\linethickness{0.4pt}
\begin{picture}(85.33,0)
(11,0)
\put(25.33,0.00){\makebox(0,0)[cc]{$V_1$}}
\put(65.33,0.00){\makebox(0,0)[cc]{$\cdots$}}
\put(85.33,0.00){\makebox(0,0)[cc]{$V_t$}}
\put(45.33,0.00){\makebox(0,0)[cc]{$V_2$}}
\put(30.33,0.00){\line(1,0){10.00}}
\put(50.33,0.00){\line(1,0){10.00}}
\put(70.33,0.00){\line(1,0){10.00}}
\put(29.83,-1.67){\line(-4,1){0.2}}
\bezier{208}(81.00,-1.67)(55.33,-7.33)(29.83,-1.67)
\end{picture}
\\[9pt]
\]
where all $V_i$ are complex vector
spaces and each line denotes
$\longrightarrow$ or $\longleftarrow$.


\end{abstract}

\renewcommand{\le}{\leqslant}
\renewcommand{\ge}{\geqslant}

\newcommand{\rank}{\mathop{\rm rank}\nolimits}

\newcommand{\fr}{1/2}
\newcommand{\frr}{3/4}
\newcommand{\lin}{\,\frac{}{\quad}\,}
\newcommand{\llin}{\,\frac{}{\quad\ }\,}
\newcommand{\lli}{\!\frac{}{\quad\ }\!}
\newcommand{\lllin}{\frac{}{\qquad}}
\newcommand{\is}{\stackrel
{\text{\raisebox{-1ex}{$\sim\
\;$}}}{\to}}

\newtheorem{theorem}{Theorem}[section]
\newtheorem{lemma}{Lemma}[section]
\newtheorem{proposition}{Proposition}[section]

\theoremstyle{remark}
\newtheorem{example}{Example}[section]

\section{Introduction}   \label{s1}

All matrices and vector spaces are
considered over the field $\mathbb C$
of complex numbers.

By the theorem on pencils of
matrices (see \cite[Sect.
XII]{gan}), every pair of
$p\times q$ matrices reduces
by transformations of
simultaneous equivalence
\begin{equation}\label{1.a}
(A_1,\,A_2)\mapsto
(R^{-1}A_1S,\,R^{-1}A_2S)
\end{equation}
($R$ and $S$ are arbitrary nonsingular
matrices) to a direct sum, determined
uniquely up to permutation of summands,
of pairs of the form
\begin{equation}\label{1.3'}
(I_n,J_n(\lambda)),\ (J_n(0),I_n),\
(F_n,G_n),\ (F_n^T,G_n^T),
\end{equation}
where
\begin{equation}\label{1.4}
F_n=\begin{bmatrix}
1&0&&0\\&\ddots&\ddots&\\0&&1&0
\end{bmatrix},\quad
G_n=\begin{bmatrix}
0&1&&0\\&\ddots&\ddots&\\0&&0&1
\end{bmatrix}, \quad n\ge 1,
\end{equation}
are $(n-1)\times n$ matrices, and
$J_n(\lambda)$ is a Jordan block. The
direct sum of pairs is defined by
\[
(A,B)\oplus(C,D)= (A\oplus C,\,B\oplus
D)=\left(\begin{bmatrix}
  A & 0 \\
  0 & C
\end{bmatrix},\: \begin{bmatrix}
  B & 0 \\
  0 & D
\end{bmatrix}\right).
\]

Note that $F_1$ and $G_1$ in
\eqref{1.4} have size $0\times 1$. It
is agreed that there exists exactly one
matrix, denoted by $0_{n0}$, of size
$n\times 0$ and there exists exactly
one matrix, denoted by $0_{0n}$, of
size $0\times n$ for every nonnegative
integer $n$; they represent the linear
mappings $0\to {\mathbb C}^n$ and
${\mathbb C}^n\to 0$ and are considered
as zero matrices. Then
$$
M_{pq}\oplus 0_{m0}=\begin{bmatrix}
  M_{pq} & 0 \\
  0 &0_{m0}
\end{bmatrix}=\begin{bmatrix}
  M_{pq}& 0_{p0} \\
  0_{mq}& 0_{m0}
\end{bmatrix}=\begin{bmatrix}
M_{pq} \\ 0_{mq}
\end{bmatrix}
$$
and
$$
M_{pq}\oplus 0_{0n}=\begin{bmatrix}
  M_{pq} & 0 \\
  0 & 0_{0n}
\end{bmatrix}=\begin{bmatrix}
  M_{pq}& 0_{pn} \\
  0_{0q}& 0_{0n}
\end{bmatrix}=\begin{bmatrix}
   M_{pq} & 0_{pn}
\end{bmatrix}
$$
for every $p\times q$ matrix $M_{pq}$.

P. Van Dooren \cite{doo} constructed an
algorithm that for every pair $(A,B)$
of $p\times q$ matrices calculates a
simultaneously equivalent pair
\begin{equation*}\label{1.4a}
  (A_1,B_1)\oplus\dots\oplus (A_r,B_r)
  \oplus (C,D),
\end{equation*}
where all $(A_i,B_j)$ are of the form
\[
(I_n,J_n(0)),\ (J_n(0),I_n),\
(F_n,G_n),\ (F_n^T,G_n^T),
\]
and the matrices $C$ and $D$ are
nonsingular. The pair $(C,D)$ is called
a {\it regular part} of $(A,B)$ and is
simultaneously equivalent to a direct
sum of pairs of the form
$(I_n,J_n(\lambda))$ with $\lambda\ne
0$. This algorithm uses only
transformations \eqref{1.a} with {\it
unitary} $R$ and $S$, which is
important for its numerical stability.

In this article we construct a unitary
algorithm for computation of the
canonical form of the system of
matrices of a chain of linear mappings
\begin{equation}\label{aa1.1}
V_1\ \frac{{\cal{A}}_1}{\qquad}\
  V_2\ \frac{{\cal{A}}_2}{\qquad}\
  \cdots\
  \frac{{\cal{A}}_{t-1}}{\qquad}\
  V_t
\end{equation}
(see Proposition \ref{t2.a}) and extend
Van Dooren's algorithm to the matrices
of a cycle of linear mappings\\[3mm]
\begin{equation}\label{1.1}
\unitlength 1.00mm
\linethickness{0.4pt}
\begin{picture}(95.00,5.00)(0,-3)
\put(-15.00,0.00){\makebox(0,0)[cc]{$\cal
A$:}}
\put(0.00,0.00){\makebox(0,0)[cc]{$V_1$}}
\put(40.00,0.00){\makebox(0,0)[cc]{$\cdots$}}
\put(60.00,0.00){\makebox(0,0)[cc]{$V_{t-1}$}}
\put(80.00,0.00){\makebox(0,0)[cc]{$V_t$}}
\put(20.00,0.00){\makebox(0,0)[cc]{$V_2$}}
\put(5.00,0.00){\line(1,0){10.00}}
\put(25.00,0.00){\line(1,0){10.00}}
\put(45.00,0.00){\line(1,0){10.00}}
\put(65.00,0.00){\line(1,0){10.00}}
\put(10.00,5.00){\makebox(0,0)[ct]{${\cal
A}_1$}}
\put(30.00,5.00){\makebox(0,0)[ct]{${\cal
A}_2$}}
\put(50.00,5.00){\makebox(0,0)[ct]{${\cal
A}_{t-2}$}}
\put(70.00,5.00){\makebox(0,0)[ct]{${\cal
A}_{t-1}$}}
\bezier{300}(4.00,-2.00)(40.00,-12.00)(76.00,-2.00)
\put(40.00,-12.00){\makebox(0,0)[cb]{${\cal
A}_t$}}
\put(83.67,-0.00){\makebox(0,0)[lc]{$,\quad
t\ge 2\,,$}}
\end{picture}
\end{equation}
\\[13pt]
(see Theorem \ref{th}), where each line
is the arrow $\longrightarrow$ or the
arrow $\longleftarrow$ and
$V_1,\dots,V_t$ are vector spaces.

For instance, the linear mappings
${\cal A}_1$ and ${\cal A}_2$ of a
cycle
\[
\unitlength 1.00mm
\linethickness{0.4pt}
\begin{picture}(90.00,5.00)(-29,0)
\put(-1.00,0.00){\makebox(0,0)[cc]{$V_1$}}
\put(2.00,1.00){\vector(1,0){26.00}}
\put(2.00,-1.00){\vector(1,0){26.00}}
\put(31.00,0.00){\makebox(0,0)[cc]{$V_2$}}
\put(15.00,4.00){\makebox(0,0)[cc]{${\cal
A}_1$}}
\put(15.00,-4.5){\makebox(0,0)[cc]{${\cal
A}_2$}}
\end{picture}\\[1em]
\]
are represented by a pair of matrices
$(A_1,A_2)$ with respect to bases in
$V_1$ and $V_2$, and a change of the
bases reduces this pair by
transformations of simultaneous
equivalence \eqref{1.a}; in this case
our algorithm coincides with Van
Dooren's algorithm.

Similarly, the linear mappings ${\cal
A}_1$ and ${\cal A}_2$ of a cycle
\[
\unitlength 1.00mm
\linethickness{0.4pt}
\begin{picture}(90.00,5.00)(30,0)
\put(59.00,0.00){\makebox(0,0)[cc]{$V_1$}}
\put(62.00,1.00){\vector(1,0){26.00}}
\put(88.00,-1.00){\vector(-1,0){26.00}}
\put(91.00,0.00){\makebox(0,0)[cc]{$V_2$}}
\put(75.00,4.00){\makebox(0,0)[cc]{${\cal
A}_1$}}
\put(75.00,-4.5){\makebox(0,0)[cc]{${\cal
A}_2$}}
\end{picture}\\*[1em]
\]
are represented by a pair $(A_1,A_2)$,
and a change of the bases in $V_1$ and
$V_2$ reduces this pair by
transformations of {\it contragredient
equivalence}
\begin{equation*}\label{1c}
(A_1,\,A_2)\mapsto
(R^{-1}A_1S,\,S^{-1}A_2R).
\end{equation*}

The {\it direct sum} of the cycle
\eqref{1.1} and a cycle
\begin{equation*}
\unitlength 1.00mm
\linethickness{0.4pt}
\begin{picture}(95.00,5.00)(-15,0)
\put(-15.00,0.00){\makebox(0,0)[cc]{${\cal
A}'$:}}
\put(0.00,0.00){\makebox(0,0)[cc]{$V'_1$}}
\put(40.00,0.00){\makebox(0,0)[cc]{$\cdots$}}
\put(60.00,0.00){\makebox(0,0)[cc]{$V'_{t-1}$}}
\put(80.00,0.00){\makebox(0,0)[cc]{$V'_t$}}
\put(20.00,0.00){\makebox(0,0)[cc]{$V'_2$}}
\put(5.00,0.00){\line(1,0){10.00}}
\put(25.00,0.00){\line(1,0){10.00}}
\put(45.00,0.00){\line(1,0){10.00}}
\put(65.00,0.00){\line(1,0){10.00}}
\put(10.00,5.00){\makebox(0,0)[ct]{${\cal
A}'_1$}}
\put(30.00,5.00){\makebox(0,0)[ct]{${\cal
A}'_2$}}
\put(50.00,5.00){\makebox(0,0)[ct]{${\cal
A}'_{t-2}$}}
\put(70.00,5.00){\makebox(0,0)[ct]{${\cal
A}'_{t-1}$}}
\bezier{300}(4.00,-2.00)(40.00,-12.00)(76.00,-2.00)
\put(40.00,-12.00){\makebox(0,0)[cb]{${\cal
A}'_t$}}
\end{picture}\\*[30pt]
\end{equation*}
with the same orientation of arrows is
the cycle ${\cal A}\oplus{\cal A}'$:
$$
\unitlength 1.00mm
\linethickness{0.4pt}
\begin{picture}(135.00,15.00)(13,-8)
\put(25.00,0.00){\makebox(0,0)[cc]{$V_1\oplus
V_1'$}}
\put(65.00,0.00){\makebox(0,0)[cc]{$V_2\oplus
V_2'$}}
\put(45.00,3.00){\makebox(0,0)[cc]{${\cal
A}_1\oplus{\cal A}'_1$}}
\put(35.00,0.00){\line(1,0){20.00}}
\put(75.00,0.00){\line(1,0){20.00}}
\put(85.00,3.00){\makebox(0,0)[cc]{${\cal
A}_2\oplus{\cal A}_2'$}}
\put(115.00,3.00){\makebox(0,0)[cc]{${\cal
A}_{t-1}\oplus{\cal A}_{t-1}'$}}
\put(100.00,0.00){\makebox(0,0)[cc]{$\cdots$}}
\put(135.00,0.00)
{\makebox(0,0)[cc]{$V_t\oplus V_t'$}}
\put(105.00,0.00){\line(1,0){20.00}}
\bezier{372}(34.33,-3.67)(80.00,-13.00)(125.67,-3.67)
\put(80.00,-12.0){\makebox(0,0)[cc]
{${\cal A}_t\oplus{\cal A}'_t$}}
\end{picture}
$$\\[-10pt]

A cycle ${\cal A}$ of the form
\eqref{1.1} is called {\it regular} if
all ${\cal A}_i$ are bijections;
otherwise it is called {\it singular}.
By a {\it regularizing decomposition}
of ${\cal A}$, we mean a decomposition
\begin{equation}\label{00}
{\cal A}={\cal D}\oplus\dots \oplus
{\cal G}\oplus{\cal P},
\end{equation}
where $\cal{D},\dots,\cal{G}$ are
direct-sum-indecomposable singular
cycles and ${\cal P}$ is a regular
cycle.

In Section \ref{s1.1} we recall notions
of quiver representations; they allow
to formulate our algorithms
pictorially.

In Section \ref{s1.2} we recall the
classification of chains \eqref{aa1.1}
and cycles \eqref{1.1} of linear
mappings. The classification of cycles
of linear mappings was obtained by
Nazarova \cite{naz} and, independently,
by Donovan and Freislich \cite{don}
(see also \cite{gab+roi}, Theorem
11.1).

In Section \ref{s4} we construct an
algorithm that gets the canonical form
of the matrices of a chain of linear
mappings using only unitary
transformations.

In Sections \ref{s1.3} and \ref{s_main}
we construct an algorithm that gets a
regularizing decomposition \eqref{00}
of a cycle of linear mappings using
only unitary transformations%
\footnote{This improves the numerical
stability of the algorithms.
Nevertheless, this does not guarantee
that the computed structure of the
cycle coincides with its original
structure.}. The singular summands
$\cal{D},\dots,\cal{G}$ will be
obtained in canonical form.

The canonical form of the (nonsingular)
matrices $P_1,\dots,P_t$ of the regular
summand
\begin{equation*}
\unitlength 1.00mm
\linethickness{0.4pt}
\begin{picture}(95.00,5.00)(-15,0)
\put(-15.00,0.00){\makebox(0,0)[cc]{$\cal
P$:}}
\put(0.00,0.00){\makebox(0,0)[cc]{$U_1$}}
\put(40.00,0.00){\makebox(0,0)[cc]{$\cdots$}}
\put(60.00,0.00){\makebox(0,0)[cc]{$U_{t-1}$}}
\put(80.00,0.00){\makebox(0,0)[cc]{$U_t$}}
\put(20.00,0.00){\makebox(0,0)[cc]{$U_2$}}
\put(5.00,0.00){\line(1,0){10.00}}
\put(25.00,0.00){\line(1,0){10.00}}
\put(45.00,0.00){\line(1,0){10.00}}
\put(65.00,0.00){\line(1,0){10.00}}
\put(10.00,5.00){\makebox(0,0)[ct]{${\cal
P}_1$}}
\put(30.00,5.00){\makebox(0,0)[ct]{${\cal
P}_2$}}
\put(50.00,5.00){\makebox(0,0)[ct]{${\cal
P}_{t-2}$}}
\put(70.00,5.00){\makebox(0,0)[ct]{${\cal
P}_{t-1}$}}
\bezier{300}(4.00,-2.00)(40.00,-12.00)(76.00,-2.00)
\put(40.00,-12.00){\makebox(0,0)[cb]{${\cal
P}_t$}}
\end{picture}
\end{equation*}
\\[13pt]
in \eqref{00} is not determined by this
algorithm. We may compute it as
follows. We first reduce $P_1$ to the
identity matrix changing the basis in
the space $U_2$. Then we reduce $P_2$
to the identity matrix changing the
basis in the space $U_3$, and so on
until obtain
\begin{equation}\label{yy}
P_1=\dots=P_{t-1}=I_n.
\end{equation}
At last, changing the bases of all
spaces $U_1,\dots,U_t$ by the same
transition matrix $S$ (this preserves
the matrices \eqref{yy}), we can reduce
the remaining matrix $P_t$ to a
nonsingular Jordan canonical matrix
$\Phi$ by similarity transformations
$S^{-1}P_tS$. Clearly, the obtained
sequence
\[
(I_n,\,\dots,\,I_n,\,\Phi)
\]
is the canonical form of the matrices
of $\cal P$.

\section{Terminology of quiver representations}
\label{s1.1}

The notion of a quiver and its
representations was introduced by
Gabriel \cite{gab} (see also
\cite[Section 7]{gab+roi}) and admits
to formulate classification problems
for systems of linear mappings. A {\it
quiver} is a directed graph; loops and
multiple arrows are allowed. Its {\it
representation} ${\cal A}$ over
$\mathbb C$ is given by assigning to
each vertex $v$ a complex vector space
$V_v$ and to each arrow $\alpha:u\to v$
a linear mapping ${\cal
A}_{\alpha}:V_u\to V_v$ of the
corresponding vector spaces.

For instance, a representation of the
quiver
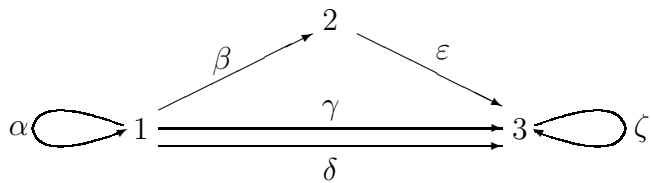
\[
\unitlength 0.60mm
\linethickness{0.4pt}
\begin{picture}(139.67,30.67)(0,5)
\put(28.67,6.67){\makebox(0,0)[cc]{1}}
\put(70.67,30.67){\makebox(0,0)[cc]{2}}
\put(1.67,6.67){\makebox(0,0)[cc]{$\alpha$}}
\put(70.67,10.67){\makebox(0,0)[cc]{$\gamma$}}
\put(46.67,21.67){\makebox(0,0)[cc]{$\beta$}}
\put(95.67,23.67){\makebox(0,0)[cc]{$\varepsilon
$}}
\put(70.67,-2.33){\makebox(0,0)[cc]{$\delta$}}
\put(32.67,6.67){\vector(1,0){76.00}}
\put(32.67,2.67){\vector(1,0){76.00}}
\put(32.67,10.67){\vector(2,1){34.00}}
\put(76.67,27.67){\vector(2,-1){32.00}}
\put(112.67,6.67){\makebox(0,0)[cc]{3}}
\put(139.67,6.67){\makebox(0,0)[cc]{$\zeta$}}
\bezier{112}(116.00,7.33)(134.67,14.33)(136.00,6.67)
\bezier{112}(116.00,6.00)(134.67,-1.00)(136.00,6.67)
\bezier{112}(24.67,7.33)(6.00,14.33)(4.67,6.67)
\bezier{112}(24.67,6.00)(6.00,-1.00)(4.67,6.67)
\put(24.67,6.00){\vector(2,1){1.00}}
\put(117.5,5.50){\vector(-2,1){2.00}}
\end{picture}\\*[6mm]
\]
is a system of linear mappings
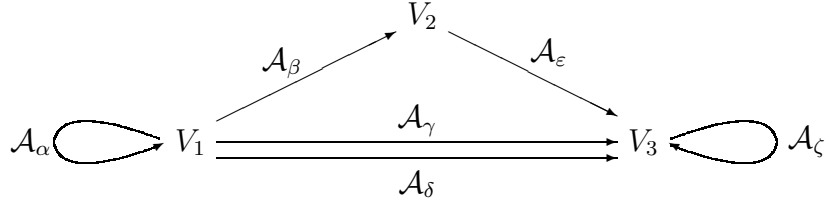
\[
\unitlength 0.70mm
\linethickness{0.4pt}
\begin{picture}(139.67,34.67)
\put(27.67,6.67){\makebox(0,0)[cc]{$V_{1}$}}
\put(71.67,30.67){\makebox(0,0)[cc]{$V_{2}$}}
\put(-2.67,6.67){\makebox(0,0)[cc]{${\cal
A}_{\alpha}$}}
\put(70.67,10.67){\makebox(0,0)[cc]{${\cal
A}_{\gamma}$}}
\put(45.0,21.67){\makebox(0,0)[cc]{${\cal
A}_{\beta}$}}
\put(95.67,23.67){\makebox(0,0)[cc]{${\cal
A}_{\varepsilon} $}}
\put(70.67,-1.33){\makebox(0,0)[cc]{${\cal
A}_{\delta}$}}
\put(32.67,6.67){\vector(1,0){76.00}}
\put(32.67,3.67){\vector(1,0){76.00}}
\put(32.67,10.67){\vector(2,1){34.00}}
\put(76.67,27.67){\vector(2,-1){32.00}}
\put(113.67,6.67){\makebox(0,0)[cc]{$V_{3}$}}
\put(144.67,6.67){\makebox(0,0)[cc]{${\cal
A}_{\zeta}$}}
\bezier{112}(119.00,7.33)(137.67,14.33)(139.00,6.67)
\bezier{112}(119.00,6.00)(137.67,-1.00)(139.00,6.67)
\bezier{112}(21.67,7.33)(3.00,14.33)(1.67,6.67)
\bezier{112}(21.67,6.00)(3.00,-1.00)(1.67,6.67)
\put(21.67,6.00){\vector(2,1){1.00}}
\put(120.5,5.50){\vector(-2,1){2.00}}
\end{picture}\\*[4mm]
\]

The number
\begin{equation*}\label{1.2c}
\dim_v{\cal A}:=\dim V_v
\end{equation*}
is called the {\it dimension of $\cal
A$ at the vertex $v$}, the set of these
numbers
\[
\dim{\cal A}:=\{\dim V_v\}_v
\]
is
called the {\it dimension of $\cal A$}.

Two representations $\cal A$ and ${\cal
A}'$ are called {\it isomorphic} if
there exists a set $\cal S$ of linear
bijections ${\cal S}_v: {\cal A}_v\to
{\cal A}'_v$ (assigned to all vertices
$v$) transforming $\cal A$ to ${\cal
A}'$. That is, the diagram
\begin{equation}\label{1.2aa}
\begin{CD}
V_u @>{\cal A}_{\alpha}>> V_v\\
 @V{\cal S}_uVV @VV{\cal S}_vV\\
 V'_u @>{\cal A}'_{\alpha}>> V'_v
\end{CD}
\end{equation}
must be commutative ($ {\cal
A}'_{\alpha}{\cal S}_u={\cal S}_v{\cal
A}_{\alpha}$) for every arrow $\alpha:
u\longrightarrow v$. In this case we
write
\begin{equation}\label{1.3}
  {\cal S}=\{{\cal S}_v\}:
  {\cal A}\is {\cal A}'
  \qquad \text{and} \qquad
  {\cal A}\simeq {\cal A}'.
\end{equation}
 The {\it direct
sum} of ${\cal A}$ and ${\cal A}'$ is
the representation ${\cal A}\oplus{\cal
A}'$ formed by $V_v\oplus V'_v$ and
${\cal A}_{\alpha}\oplus{\cal
A}'_{\alpha}$.

The following theorem is a well-known
corollary of the Krull--Schmidt theorem
\cite[Theorem I.3.6]{bas} and holds for
representations over an arbitrary
field.

\begin{theorem} \label{t.0}
Every representation of a quiver
decomposes into a direct sum of
indecomposable representations
uniquely, up to isomorphism of
summands.
\end{theorem}

Every representation of a quiver over
${\mathbb C}$ is isomorphic to a
representation, in which the vector
spaces $V_v$ assigned to the vertices
all have the form ${\mathbb
C}\oplus\dots\oplus {\mathbb C}$. Such
a representation of dimension $\{d_v\}$
with $d_v\in\{0,1,2,\ldots\}$ is
called a {\it matrix representation}%
 \footnote{A matrix representation also
arises when we fix bases in all the
spaces of a representation. As follows
from \eqref{1.2aa}, two matrix
representations are isomorphic if and
only if they give the same
representation but in possible
different bases.}
 and is given by a set
$\mathbb A$ of matrices ${\mathbb
A}_{\alpha}\in {\mathbb C}^{d_v\times
d_u}$ assigned to the arrows
$\alpha:u\longrightarrow v$. We will
consider mainly matrix representations.

For every matrix representation
$\mathbb{A} =\{A_{\alpha}\}$ of a
quiver $\cal Q$, we define the {\it
transpose matrix representation}
\begin{equation}\label{1.xy}
\mathbb{A}^T=\{A_{\alpha}^T\}
\end{equation}
of the quiver ${\cal Q}^T$ obtained
from $\cal Q$ by changing the direction
of each arrow. Clearly,
\begin{equation}\label{1.yx}
{\mathbb
S}=\{S_v\}:\,\mathbb{A}\is\mathbb{B}
\qquad\text{implies}\qquad {\mathbb
S}^T=\{S_v^T\}:\,\mathbb{B}^T\is
\mathbb{A}^T
\end{equation}

The systems of linear mappings
\eqref{aa1.1} and \eqref{1.1} may be
considered as representations of the
quivers
\begin{equation} \label{x1.5}
{\cal{L}} :\qquad 1\
\frac{{\alpha}_1}{\qquad}\
  2\ \frac{{\alpha}_2}{\qquad}\
  \cdots\
  \frac{{\alpha}_{t-2}}{\qquad}\
  {(t-1)}
  \frac{{\alpha}_{t-1}}{\qquad}\
  t
\end{equation}
and\\
\begin{equation}\label{1.2}
\unitlength 1.00mm
\linethickness{0.4pt}
\begin{picture}(90.00,4.00)
(-7,-1)
\put(-13.00,0.00){\makebox(0,0)[cc]{${\cal
C}:$}}
\put(0.00,0.00){\makebox(0,0)[cc]{$1$}}
\put(40.00,0.00){\makebox(0,0)[cc]{$\cdots$}}
\put(65.00,0.00){\makebox(0,0)[cc]{$(t-1)$}}
\put(90.00,0.00){\makebox(0,0)[cc]{$t$}}
\put(20.00,0.00){\makebox(0,0)[cc]{$2$}}
\put(5.00,0.00){\line(1,0){10.00}}
\put(25.00,0.00){\line(1,0){10.00}}
\put(10.00,4.00){\makebox(0,0)[ct]{${\alpha}_1$}}
\put(30.00,4.00){\makebox(0,0)[ct]{${\alpha}_2$}}
\put(50.00,4.00){\makebox(0,0)[ct]{${\alpha}_{t-2}$}}
\put(80.00,4.00){\makebox(0,0)[ct]{${\alpha}_{t-1}$}}
\put(45.00,-11.00){\makebox(0,0)[cb]{${\alpha}_t$}}
\put(45.00,0.00){\line(1,0){10.00}}
\put(75.00,0.00){\line(1,0){10.00}}
\bezier{336}(86.00,-2.33)(45.00,-10.33)(4.00,-2.33)
\end{picture}\\*[35pt]
\end{equation}
with the same orientations of arrows as
in \eqref{aa1.1} and \eqref{1.1}. The
quiver \eqref{1.2} will be called a
{\it cycle}; the symbol $\cal C$ will
always denote the cycle \eqref{1.2}.

If $\mathbb A$ is a matrix
representation of a quiver with an
indexed set of arrows
$\{\alpha_i\,|\,i\in I\}$, we will
write $A_i$ instead of $\mathbb
A_{\alpha_i}$. So a matrix
representation $\mathbb A$ of the cycle
$\cal C$ is given by a sequence of
matrices
\begin{equation*}\label{1.3a}
  \mathbb A=(A_1,\dots,A_t).
\end{equation*}

\section{Classification theorems} \label{s1.2}

In this section, we recall the
classification of representations of
the quivers \eqref{x1.5} and
\eqref{1.2}, and mention articles
considering special cases. Some of
these articles are little known outside
of representation theory.

We first consider the cycles of length
2. The representations of the cycle
$1\rightrightarrows 2$ were classified
by Kronecker \cite{kro} in 1890 (see
also \cite[Sect. V]{gan} or \cite[Sect.
1.8]{gab+roi}): every pair of $p\times
q$ matrices is simultaneously
equivalent to a direct sum of pairs of
the form \eqref{1.3'}. A simple and
short proof of this result was obtained
by Nazarova and Roiter \cite{naz+roi}.

A classification of representations of
the cycle $1\rightleftarrows 2$ was
obtained by Dobrovol$'$skaya and
Ponomarev \cite{dob+pon} in 1965: every
matrix representation is isomorphic to
a direct sum, determined uniquely up to
permutation of summands, of matrix
representations of the form
\begin{equation}\label{1.5}
(I_n,J_n(\lambda)),\ (J_n(0),I_n),\
(F_n,G_n^T),\ (F_n^T,G_n)
\end{equation}
(see \eqref{1.4}). Over an arbitrary
field, the Jordan block $J_n(\lambda)$
is replaced by a Frobenius block
$$
\Phi_n=\begin{bmatrix}
0&1&&\\&\ddots&\ddots&\\&&0&1\\
-\alpha_n& -\alpha_{n-1} &\cdots&
-\alpha_1
\end{bmatrix},
$$
where
\[
x^n+\alpha_1 x^{n-1}+\dots+
\alpha_{n-1}x+\alpha_n= p(x)^t
\]
for some irreducible
polynomial $p(x)$ and some integer $t$.
This result was proved again by
Rubi\'{o} and Gelonch \cite{rub+gel} in
1992, Olga Holtz \cite{hol} in 2000,
and Horn and Merino \cite{hor+mer} in
1995; the last article also contains
many applications of this
classification.

A classification of systems of linear
mappings of the form
$$
\begin{array}{ccc}
      V_1 & \longrightarrow & V_2\\
  \downarrow& &\uparrow \\
  V_3 & \longleftarrow & V_4
\end{array}
$$
was given by Nazarova \cite{naz1} in
1961 over the field with two elements,
and by Nazarova \cite{naz2} in 1967
over an arbitrary field.

A quiver is said to be of {\it tame
type} if the problem of classifying its
representations does not contain the
problem of classifying pairs of
matrices up to simultaneous similarity.
If a quiver $\cal Q$ is not of tame
type, then a full classification of its
representations is impossible since it
must contain a classification of
representations of all quivers, see
\cite[Sect. 3.1]{ser} or \cite[Sect.
2]{bel-ser}. Nevertheless, each
particular representation of $\cal Q$
can be reduced to canonical form, see
\cite{bel1} or \cite[Sect. 1.4]{ser}.

Nazarova \cite{naz} and, independently,
Donovan and Freislich \cite{don} in
1973 classified representations of all
quivers of tame type (see also
\cite[Sect. 11]{gab+roi}). In
particular, they classified
representations of the cycle
\eqref{1.2}, which is of tame type (see
this classification also in
\cite[Theorem 11.1]{gab+roi}). This
classification is not mentioned in many
articles on linear algebra and system
theory that study its special cases
(for instance, in the article by
Gelonch \cite{gel} containing the
classification of representations of
the cycle \eqref{1.2} with orientation
$1\to 2\to \dots \to t\to 1$).

Gabriel \cite{gab} (see also
\cite[Sect. 11]{gab+roi}) classified
representations of all quivers having a
finite number of nonisomorphic
indecomposable representations. In
particular, he classified
representations of the quiver
\eqref{x1.5}.
\medskip

Now we formulate theorems that classify
representations of the quivers
\eqref{x1.5} and \eqref{1.2}.

For every pair of integers $(i,j)$ such
that $1\le i\le j\le t$, we define the
matrix representation
\begin{equation}\label{x1.4}
{\mathbb{L}}_{ij} :\qquad 1\
\frac{0}{\qquad}\ \cdots \
\frac{0}{\qquad}\ i\
\frac{I_{1}}{\qquad}\ \cdots\
\frac{I_{1}}{\qquad}\ j\
\frac{0}{\qquad}\ \cdots\
\frac{0}{\qquad}\ t
\end{equation}
of dimension
$(0,\dots,0,1,\dots,1,0,\dots 0)$ of
the quiver \eqref{x1.5}. By the next
theorem, which holds over an arbitrary
field, the representations
${\mathbb{L}}_{ij}$ form a full set of
nonisomorphic indecomposable matrix
representations of \eqref{x1.5}.

\begin{theorem}[see \cite{gab}]
\label{xt1.1} For every system of
linear mappings \eqref{aa1.1}, there
are bases of the spaces
$V_1,\dots,V_t$, in which the sequence
of matrices of
${\cal{A}}_{1},\dots,{\cal{A}}_{t-1}$
is a direct sum of sequences
$(0,\dots,0,I_{1},\dots,I_{1},0,\dots,0)$
of dimension
$(0,\dots,0,1,\dots,1,0,\dots 0)$. This
sum is determined by the system
\eqref{aa1.1} uniquely up to
permutation of summands.
\end{theorem}

The classification of representations
of a cycle \eqref{1.2} follows from
Theorem \ref{t.0} and the next fact: if
a matrix representation of this cycle
is direct-sum-indecomposable, then at
least $t-2$ of its matrices are
nonsingular. Clearly, these $t-2$
matrices reduce to the identity
matrices and the remaining two matrices
reduce to the form \eqref{1.3'} or
\eqref{1.5} depending on the
orientation of their arrows. This gives
the following theorem.

\begin{theorem}[see \cite{don} or \cite{naz}]
\label{t1.1} For every system of linear
mappings \eqref{1.1}, there are bases
in the spaces $V_1,\dots,V_t$, in which
the sequence of matrices of
${\cal{A}}_{1},\dots,{\cal{A}}_{t}$ is
a direct sum, determined by \eqref{1.1}
uniquely up to permutation of summands,
of sequences of the following form
$($the points denote sequences of
identity matrices or $0_{00})$:
\begin{itemize}

\item[\rm{(i)}] $(J_n(\lambda),\ldots)$
with $\lambda\ne 0$;

\item[\rm{(ii)}] $(\ldots, J_n(0), \ldots)$ with
$J_n(0)$ at the place $i\in \{1,\ldots,
t\}$;

\item[\rm{(iii)}] $(\ldots, A_i,\ldots,
A_j,\ldots)$, where $A_i$ and $A_j$
depend on the direction of the mappings
${\cal A}_i$ and ${\cal A}_j$ in the
sequence
\begin{equation*}\label{1.6}
  V_1\ \frac{{\cal A}_1}{\qquad}\
  V_2\ \frac{{\cal A}_2}{\qquad}\
  \cdots\
  \frac{{\cal A}_{t-1}}{\qquad}\
  V_t\ \frac{{\cal A}_t}{\qquad}\ V_1
\end{equation*}
$($see \eqref{1.1}$)$ as follows:
\end{itemize}
\[
(A_i,A_j)=
  \begin{cases}
 (F_n,G_n) \text{ or }
(F_n^T,G_n^T) & \text{if ${\cal A}_i$
and ${\cal A}_j$ have opposite
directions}, \\
 (F_n,G_n^T)\text{ or }
 (F_n^T,G_n) & \text{otherwise}.
  \end{cases}
\]
\end{theorem}

This theorem, with a nonsingular
Frobenius block instead of
$J_n(\lambda)$ in (i), holds over an
arbitrary field.\medskip

In the remaining part of this section,
we recall Gabriel and Roiter's
construction \cite[Sect. 11.1]{gab+roi}
of summands (ii) and (iii).

For every integer $n$, denote by $[n]$
the natural number such that
\begin{equation*}\label{1'.3}
1\le [n]\le t\quad \text{and} \quad
[n]\equiv n \bmod t\,.
\end{equation*}

 Let
\begin{equation}\label{1'.1}
  l\llin
  (l+1)\llin
  (l+2)\,
 \llin\cdots\llin\,
  r, \qquad 1\le l\le t,
\end{equation}
be a ``clockwise walk'' on the cycle
\eqref{1.2}) that starts at the vertex
$l$, passes through the vertices
\[
[l+1],\ [l+2],\: \dots\:,\ [r-1],
\]
and stops at the vertex $[r]$. This
walk determines the representation
$\cal A$ of $\cal C$ in which each
space ${V}_v$ is spanned by all
$i\in\{l,l+1,\dots,r\}$ such that
$[i]=v$:
$$
{V}_v=\langle i\,|\,l\le i\le r,\ [i]=v
\rangle,
$$
and all the nonzero actions of linear
mappings ${\cal
A}_{\alpha_1},\dots,{\cal
A}_{\alpha_t}$ on the basis vectors are
given by \eqref{1'.1}. The matrices of
${\cal A}_{\alpha_1},\dots,{\cal
A}_{\alpha_t}$ in these bases form a
matrix representation denoted by
\begin{equation}\label{ppp}
 \mathbb G_{lr}.
\end{equation}

\begin{example}\label{e1.a}
The walk
$$
\unitlength 0.50mm
\linethickness{0.4pt}
\begin{picture}(83.00,36.67)(-9,5)
\put(17.00,27.67){\makebox(0,0)[cc]{1}}
\put(22.00,29.67){\vector(2,1){10.00}}
\put(37.00,36.67){\makebox(0,0)[cc]{2}}
\put(57.00,36.67){\vector(-1,0){15.00}}
\put(62.00,36.67){\makebox(0,0)[cc]{3}}
\put(83.00,22.67){\makebox(0,0)[cc]{4}}
\put(79.00,26.00){\vector(-3,2){13.33}}
\put(79.00,20.00){\vector(-3,-2){13.33}}
\put(17.00,17.34){\makebox(0,0)[cc]{$7$}}
\put(22.00,20.34){\vector(2,1){10.00}}
\put(37.00,27.33){\makebox(0,0)[cc]{$8$}}
\put(22.00,15.33){\vector(2,-1){10.00}}
\put(37.00,8.00){\makebox(0,0)[cc]{6}}
\put(57.00,8.00){\vector(-1,0){15.00}}
\put(62.00,8.00){\makebox(0,0)[cc]{5}}
\put(57.00,27.33){\vector(-1,0){15.00}}
\put(62.00,27.33){\makebox(0,0)[cc]{$9$}}
\end{picture}
$$
on the cycle
$$
\unitlength 0.5mm \linethickness{0.4pt}
\begin{picture}(95.33,35.00)(0,10)
\put(30.00,24.67){\makebox(0,0)[cc]{1}}
\put(35.00,27.67){\vector(2,1){10.00}}
\put(50.00,35.00){\makebox(0,0)[cc]{2}}
\put(70.00,35.00){\vector(-1,0){15.00}}
\put(75.00,35.00){\makebox(0,0)[cc]{3}}
\put(95.33,24.67){\makebox(0,0)[cc]{4}}
\put(90.33,27.67){\vector(-2,1){10.00}}
\put(35.00,22.33){\vector(2,-1){10.00}}
\put(50.00,15.00){\makebox(0,0)[cc]{6}}
\put(70.00,15.00){\vector(-1,0){15.00}}
\put(75.00,15.00){\makebox(0,0)[cc]{5}}
\put(90.33,22.33){\vector(-2,-1){10.00}}
\put(11.33,25.00){\makebox(0,0)[cc]{$\cal
C$:}}
\end{picture}
$$
determines the representation
$$
\unitlength 1.00mm
\linethickness{0.4pt}
\begin{picture}(99.33,39.00)(0,5)
\put(26.00,24.67){\makebox(0,0)[cc]
{$\langle{}1,7\rangle$}}
\put(32.33,27.67){\vector(2,1){10.00}}
\put(48.67,35.00){\makebox(0,0)[cc]
{$\langle{}2,8\rangle$}}
\put(70.00,35.00){\vector(-1,0){15.00}}
\put(76.33,35.00){\makebox(0,0)[cc]
{$\langle{}3,9\rangle$}}
\put(99.33,24.67){\makebox(0,0)[cc]
{$\langle{}4\rangle$}}
\put(93.00,27.67){\vector(-2,1){10.00}}
\put(32.33,22.33){\vector(2,-1){10.00}}
\put(48.67,15.00){\makebox(0,0)[cc]
{$\langle{}6\rangle$}}
\put(70.00,15.00){\vector(-1,0){15.00}}
\put(76.33,15.00){\makebox(0,0)[cc]
{$\langle{}5\rangle$}}
\put(93.00,22.33){\vector(-2,-1){10.00}}
\put(7.33,25.00){\makebox(0,0)[cc]
{${\mathbb{G}}_{1,9}$:}}
\put(33.00,32.00){\makebox(0,0)[cc]{$I_2$}}
\put(63.00,38.00){\makebox(0,0)[cc]{$I_2$}}
\put(93.00,35.00){\makebox(0,0)[cc]{$
 \begin{bmatrix}
 1\\0\end{bmatrix}
 $}}
\put(93.00,19.00){\makebox(0,0)[cc]{$I_1$}}
\put(63.00,12.00){\makebox(0,0)[cc]{$I_1$}}
\put(32.00,17.00){\makebox(0,0)[cc]%{$[0\,1]$}}
{$\begin{bmatrix} 0&\!\!\!
1\end{bmatrix}$}}
\end{picture}\\*[-3mm]
$$
\end{example}

\begin{lemma}[see {\cite[Sect.
11.1]{gab+roi}}] \label{l00}
 The set of all\/
$\mathbb G_{lr}$ coincides with the set
of matrix representations of the form
{\rm(ii)} and {\rm(iii)}:

\begin{itemize}
\item[{\rm(a)}]
$\mathbb G_{lr}$ with $r\not\equiv l-1
\bmod t$ is the matrix representation
{\rm(iii)} of dimension
$(d_1,\dots,d_t)$, where $d_i$ is the
number of $n\in\{l, l+1, \dots,r\}$
such that $[n]=i$. {\rm(}Note that all
representations of the form {\rm(iii)}
have distinct dimensions and so they
are determined by their
dimensions.{\rm)}

\item[{\rm(b)}] $\mathbb G_{l,l-1+pt}=
(I_p,\dots,I_p, J_p(0),
I_p,\dots,I_p)$, where $J_p(0)$ is at
the $[l-1]$-st place.
\end{itemize}
\end{lemma}

We will use the following notation. If
all arrows in a representation
\begin{equation}\label{x2.1aac}
\unitlength 0.4mm \linethickness{0.4pt}
\begin{picture}(25.00,25.00)(-20,12)
\put(13.00,24.00){\makebox(0,0)[lc]
{$\scriptstyle{ A_1}$}}
\put(3.00,18.00){\makebox(0,0)[lc]
{$\scriptstyle{ A_2}$}}
\put(13.00,-1.00){\makebox(0,0)[lc]
{$\scriptstyle{ A_n}$}}
\put(0.00,30.00){\makebox(0,0)[rc]
{$u_1$}}
\put(0.00,15.00){\makebox(0,0)[rc]
{$u_2$}}
\put(0.00,5.00){\makebox(0,0)[rc]{$\cdots$}}
\put(0.00,-5.00){\makebox(0,0)[rc]
{$u_n$}}
\put(25.00,10.00){\makebox(0,0)[lc]{$v$}}
\put(3.00,14.00){\line(5,-1){18.00}}
\put(3.00,29.00){\line(5,-4){18.00}}
\put(3.00,-4.00){\line(5,3){18.00}}
\end{picture}
\end{equation}\\[0.5em]
have the same orientation, then instead
of \eqref{x2.1aac} we will write
\begin{equation}\label{x2.1acc}
\unitlength 0.4mm \linethickness{0.4pt}
\begin{picture}(25.00,25.00)(-20,12)
\put(13.00,29.00){\makebox(0,0)[lc]
{\fbox{$\scriptstyle{A}$}}}
\put(0.00,30.00){\makebox(0,0)[rc]
{$u_1$}}
\put(0.00,15.00){\makebox(0,0)[rc]
{$u_2$}}
\put(0.00,5.00){\makebox(0,0)[rc]{$\cdots$}}
\put(0.00,-5.00){\makebox(0,0)[rc]
{$u_n$}}
\put(25.00,10.00){\makebox(0,0)[lc]{$v$}}
\put(3.00,15.00){\line(5,-1){18.00}}
\put(3.00,29.00){\line(5,-4){18.00}}
\put(3.00,-4.00){\line(5,3){18.00}}
\end{picture}
\end{equation}\\[0.5em]
where
\begin{equation}\label{2.1.5aq}
 A=\begin{cases}
 [\,A_1\,|\dots|\,A_n\,]& \text{if
 $u_1\longrightarrow v$,
$u_2\longrightarrow v,\dots,
u_n\longrightarrow v$}, \\[0.1em]
 \left[\begin{tabular}{c}
    $A_1$\\ \hline $\cdots$\\ \hline
    $A_n$
\end{tabular}\right] &
\text{if $u_1\longleftarrow v$,
$u_2\longleftarrow v,\dots,
u_n\longleftarrow v$}.
  \end{cases}
\end{equation}
The partition of $A$ into strips is
fully determined by the dimensions of
\eqref{x2.1aac} at the vertices
$u_1,\dots,u_n$.

\section{Chains of linear mappings}
\label{s4}

In this section we give an algorithm
that calculates the canonical form of
the matrices of a chain of linear
mappings \eqref{aa1.1} using only
unitary transformations.

Choosing bases in the spaces
$V_1,\dots,V_t$, we may represent a
system of linear mappings \eqref{aa1.1}
by the sequence of matrices
${\mathbb{A}}=(A_1,\dots,A_{t-1})$. We
will consider this sequence as a matrix
representation
\begin{equation}\label{x1.2}
{\mathbb{A}} :\qquad 1\
\frac{{{A}}_1}{\qquad}\
  2\ \frac{{{A}}_2}{\qquad}\
  \cdots\
  \frac{{{A}}_{t-1}}{\qquad}\
  t\,.
\end{equation}
of the quiver \eqref{x1.5}.

For every vertex $i$, a change of the
basis in $V_i$ changes $\mathbb{A}$.
This transformation of $\mathbb{A}$
will be called a {\it transformation at
the vertex} $i$. It will be called a
{\it unitary transformation} if the
transition matrix to a new basis of
$V_i$ is unitary.

\subsection*{The algorithm for chains:}
  %\label{xs2}

Let $\mathbb{A}$ be a matrix
representation \eqref{x1.2} of
dimension
\begin{equation*}\label{x2.a}
  \dim \mathbb{A}= (d_1,\dots,d_t)
\end{equation*}
of the quiver \eqref{x1.5}. We will
sequentially split $\mathbb{A}$ into
representations of the form
\eqref{x1.4}.
 \medskip

\noindent {\bf Step 1:} By unitary
transformations at vertices 1 and 2, we
reduce $A_1$ to the form
\begin{equation}\label{x2.1'}
B_1=
\begin{bmatrix}
0&H\\0&0
\end{bmatrix},
\end{equation}
where $H$ is a nonsingular $k\times k$
matrix. These transformations change
$A_2$; denote the new matrix by $A_2'$.
Denote also by ${\cal P}_1$ the set
consisting of $d_1-k$ representations
of the form ${\mathbb L}_{11}$.

Next we will transform the
representation $A$ into a
representation $M_1$ of a ``split''
quiver depending on the direction of
$\alpha_1$ in \eqref{x1.5} as follows:
\medskip

{\it Case $\alpha_1:1\longrightarrow
2$}.\ \ Then
\\[0.8em]
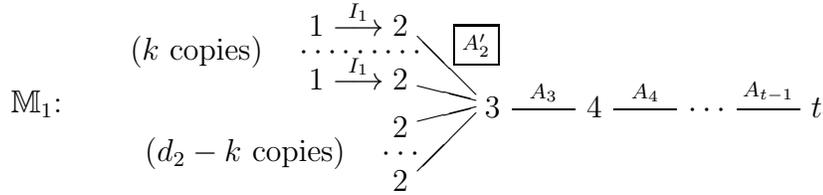
\begin{equation*}\label{x2.1aa}
 \unitlength 0.6mm
\linethickness{0.4pt}
\begin{picture}(15.00,17.00)(5,-1)
\put(8.00,13.00){\makebox(0,0)[lc]
{\fbox{$\scriptstyle{ A_2'}$}}}
\put(-90.00,0.00){\makebox(0,0)[lc]{${\mathbb
M}_1$:}}
\put(15.00,0.00){\makebox(0,0)[lc]{$3
\stackrel{A_3}{\lllin} {4}
\stackrel{A_4}{\lllin}
  \cdots
\stackrel{A_{t-1}}{\lllin}
  t$}}
\put(-2.00,-11.00)
{\makebox(0,0)[rc]{${%\scriptstyle
(d_2-k \text{ copies})}
\quad\cdots\!\!$}}
\put(13.00,-2.00){\line(-1,-1){13.00}}
\put(-2.00,-17.00){\makebox(0,0)[rc]{$2$}}
 \put(2.00,11.50){\makebox(0,0)[rc]
 {${%\scriptstyle
 (k \text{ copies})}
\quad\cdots\cdots\cdots $}}
\put(13.00,2.00){\line(-1,1){13.00}}
\put(-2.00,19.00){\makebox(0,0)[rc]
{$1\stackrel{I_{1}}{\longrightarrow}2$}}
\put(0.00,-4.00){\line(4,1){13.00}}
\put(-2.00,-5.00){\makebox(0,0)[rc]{$2$}}
\put(0.00,4.00){\line(4,-1){13.00}}
\put(-2.00,7.00){\makebox(0,0)[rc]
{$1\stackrel{I_{1}}{\longrightarrow}2$}}
\end{picture}\\*[3em]
\end{equation*}
(there are $k$ fragments of the form
$1{\longrightarrow} 2\llin 3$ and
$d_2-k$ fragments of the form $2\llin
3$). The direction of the arrows is the
same as in the quiver \eqref{x1.5}.
\medskip

{\it Case $\alpha_1:1\longleftarrow
2$}. Then\\[0.8em]
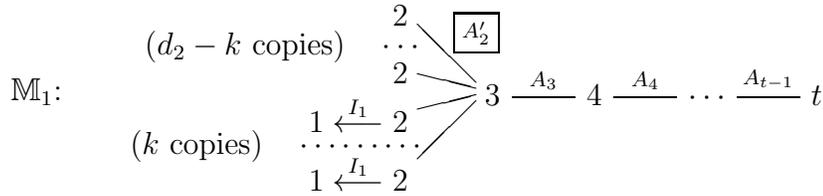
\begin{equation*}\label{x2.1ab}
 \unitlength 0.6mm
\linethickness{0.4pt}
\begin{picture}(15.00,17.00)(5,-2)
\put(8.00,13.00){\makebox(0,0)[lc]
{\fbox{$\scriptstyle{ A_2'}$}}}
\put(-90.00,0.00){\makebox(0,0)[lc]{${\mathbb
M}_1$:}}
\put(15.00,0.00){\makebox(0,0)[lc]{$3
\stackrel{A_3}{\lllin} {4}
  \stackrel{A_4}{\lllin}
  \cdots
  \stackrel{A_{t-1}}{\lllin}
  t$}}
\put(2.00,-12.00){\makebox(0,0)[rc]
{${%\scriptstyle
 (k \text{ copies})}
\quad\cdots\cdots\cdots$}}
\put(13.00,-2.00){\line(-1,-1){13.00}}
\put(-2.00,-18.00){\makebox(0,0)[rc]
{$1\stackrel{I_{1}}{\longleftarrow}2$}}
\put(-2.00,10.00){\makebox(0,0)[rc]
{${%\scriptstyle
 (d_2-k \text{ copies})}
\quad\cdots\!\!$}}
\put(13.00,2.00){\line(-1,1){13.00}}
\put(-2.00,17.00){\makebox(0,0)[rc]{$2$}}
\put(0.00,-4.00){\line(4,1){13.00}}
\put(-2.00,-5.00){\makebox(0,0)[rc]
{$1\stackrel{I_{1}}{\longleftarrow}2$}}
\put(0.00,4.00){\line(4,-1){13.00}}
\put(-2.00,4.00){\makebox(0,0)[rc]{$2$}}
\end{picture}\\*[3.5em]
\end{equation*}

\noindent {\bf Step  $\pmb{r\ (1< r<
t)}$:}\ \ Assume we have constructed in
the step $r-1$ the set ${\cal P}_{r-1}$
consisting of representations of the
form ${\mathbb L}_{ij}$, $1\le i\le
j<r$, and a quiver representation
${\mathbb M}_{r-1}$:\\[0.5em]
\begin{equation}\label{x2.1c}
\unitlength 0.4mm \linethickness{0.4pt}
\begin{picture}(25.00,25.00)(-20,7)
\put(13.00,29.00){\makebox(0,0)[lc]
{\fbox{$\scriptstyle{ A_r'}$}}}
\put(0.00,30.00){\makebox(0,0)[rc]
{${%\scriptstyle
 (k_{1} \text{ copies})}
 \quad p_1
\stackrel{I_{1}}{\lli} (p_1+1)
 \stackrel{I_{1}}{\lli}\cdots
 \stackrel{I_{1}}{\lli}  r$}}
\put(0.00,15.00){\makebox(0,0)[rc]
{${%\scriptstyle
 (k_{2}\text{ copies})}
 \quad p_2
 \stackrel{I_{1}}{\lli}(p_2+1)
 \stackrel{I_{1}}{\lli}
 \cdots
 \stackrel{I_{1}}{\lli}  r$}}
\multiput(-172.00,3.00)(5,0){35}{$\cdot$}
\put(0.00,-5.00){\makebox(0,0)[rc]
{${%\scriptstyle
 (k_{r}\text{ copies})}
\quad p_r
 \stackrel{I_{1}}{\lli}
(p_r+1)\stackrel{I_{1}}{\lli}
 \cdots
\stackrel{I_{1}}{\lli}  r$}}
\put(25.00,10.00){\makebox(0,0)[lc]{$(r+1)\,
\stackrel{A_{r+1}}{\lllin}
  \cdots
  \stackrel{A_{t-1}}{\lllin}
  t$}}
\put(3.00,15.00){\line(5,-1){18.00}}
\put(3.00,29.00){\line(5,-4){18.00}}
\put(3.00,-4.00){\line(5,3){18.00}}
\end{picture}\\*[2.2em]
\end{equation}
in which every
\[
p_i\stackrel{I_{1}}{\llin}
(p_i+1)\stackrel{I_{1}}{\llin}
\cdots\stackrel{I_{1}}{\llin} r
\]
repeats $k_i$ times,
$k_1+\dots+k_r=d_r$, all $k_i\ge 0$,
and
\[\{p_1,p_2,\dots,p_r\}=\{1,2,\dots,r\}.\]
The direction of the arrows is the same
as in the quiver \eqref{x1.5}.
\medskip

{\it Case $\alpha_r:r\longrightarrow
r+1$} (see \eqref{x1.5}).\ \ We divide
$A'_r$ into $r$ vertical strips of
sizes $k_1,k_2,\dots,k_r$ and reduce
$A_r'$ to the form
\begin{equation}\label{x2.6'}
B_r=
\left[\begin{tabular}{cc|cc|cc|c|cc}
 0&$H_{1}$ &$*$&$*$ &$*$&$*$ &$\cdots$ &$*$&$*$\\
 0&0   &0&$H_{2}$   &$*$&$*$ &$\cdots$ &$*$&$*$\\
 0&0   &0&0   &0&$H_{3}$ &$\cdots$ &$*$&$*$\\
 $\cdots$&$\cdots$ &$\cdots$&$\cdots$
&$\cdots$&$\cdots$ &$\cdots$
&$\cdots$&$\cdots$\\
    0&0 &0&0 &0&0 &$\cdots$ &0&$H_{r}$\\
    0&0 &0&0 &0&0 &$\cdots$ &0&0\\
\end{tabular}\right]
\end{equation}
(where each $H_{i}$ is a nonsingular
$l_i\times l_i$ matrix and each $*$ is
an unspecified matrix) starting from
the first vertical strip by unitary
column-transformations within vertical
strips and by unitary
row-transformations. These
row-transformations are transformations
at the vertex $r+1$ and they change
$A_{r+1}$; denote the obtained matrix
by $A_{r+1}'$. Denote also by ${\cal
P}_{r}$ the set obtained from ${\cal
P}_{r-1}$ by including $k_i-l_i$
representations of the form ${\mathbb
L}_{p_ir}$ for all $i=1,\dots,r$.
Construct the quiver representation
\begin{equation*}\label{x2.1d}
\unitlength 0.4mm \linethickness{0.4pt}
\begin{picture}(25.00,20.00)(-50,20)
\put(-190.00,10.00){\makebox(0,0)[rc]
{${\mathbb M}_{r}$:}}
\put(12.00,28.00){\makebox(0,0)[lc]
{\fbox{$\scriptstyle{ A_{r+1}'}$}}}
\put(0.00,25.00){\makebox(0,0)[rc]
{${%\scriptstyle
 (l_{1} \text{ copies})}
\quad p_1
 \stackrel{I_{1}}{\lli}
 \cdots
 \stackrel{I_{1}}{\lli} (r+1)$}}
\multiput(-143.00,8.00)(5,0){29}{$\cdot$}
\put(0.00,-5.00){\makebox(0,0)[rc]
{${%\scriptstyle
 (l_{r} \text{ copies})}
 \quad p_r
 \stackrel{I_{1}}{\lli}
 \cdots
 \stackrel{I_{1}}{\lli} (r+1)$}}
\put(25.00,10.00){\makebox(0,0)[lc]{$(r+2)\,
 \stackrel{A_{r+2}}{\lllin}
  \cdots
  \stackrel{A_{t-1}}{\lllin}
  t$}}
\put(3.00,24.00){\line(5,-3){18.00}}
\put(3.00,-4.00){\line(5,3){18.00}}
\put(3.00,-22.00){\line(3,4){18.00}}
\put(0.00,-27.00){\makebox(0,0)[rc]
{${%\scriptstyle
 (d_{r+1}-l_1-\dots-l_{r} \text{
copies})} \quad \, (r+1)$}}
\end{picture}
\end{equation*}\\[4.3em]
(Hence, $k_i-l_i$ representations
\[{\mathbb L}_{p_ir}:\ p_i
\stackrel{I_{1}}{\lllin}
  (p_i+1)\stackrel{I_{1}}{\lllin}
  \cdots \stackrel{I_{1}}{\lllin} r
\]
for each $i=1,\dots,r$ ``break
away'' from the representation
\eqref{x2.1c} and join to the set
${\cal P}_{r-1}$.) In particular, if
$r=t-1$, then ${\mathbb M}_{r}$ takes
the form\\[2mm]
\begin{equation}\label{x2.1da}
\unitlength 0.4mm \linethickness{0.4pt}
\begin{picture}(25.00,25.00)(-100,0)
\put(0.00,25.00){\makebox(0,0)[rc]
{${%\scriptstyle
 (l_{1}\text{ copies})}
\quad p_1
 \stackrel{I_{1}}{\lli}
 \cdots
 \stackrel{I_{1}}{\lli}  t$}}
\put(-180.00,0.00){\makebox(0,0)[rc]{${\mathbb
M}_{t-1}:$}}
\multiput(-120.00,8.00)(5,0){24}{$\cdot$}
\put(0.00,-5.00){\makebox(0,0)[rc]
{${%\scriptstyle
 (l_{t-1}\text{
copies})} \quad p_{t-1}
 \stackrel{I_{1}}{\lli} \cdots
 \stackrel{I_{1}}{\lli}  t$}}
\put(0.00,-27.00){\makebox(0,0)[rc]
{${%\scriptstyle
 (d_t-l_1-\dots-l_{t-1}
\text{ copies})}\quad \, t$}}
\end{picture}\\*[15mm]
\end{equation}

{\it Case $\alpha_r:r\longleftarrow
r+1$}.\ \ We partition $A'_r$ into $r$
horizontal strips of sizes
$k_1,k_2,\dots,k_{r}$ and reduce $A_r'$
to the form
\begin{equation}\label{x2.7'}
B_r= {\left[\begin{tabular}{ccccccc}
    0&$H_{1}$&$*$  &$\cdots$ &$*$&$*$&$*$\\
    0&0&$*$  &$\cdots$ &$*$&$*$&$*$\\
     \hline\multicolumn{7}{c}
  {$\!\!\!\cdots\cdots\cdots\cdots\cdots\cdots
  \cdots\cdots\cdots\cdots\cdots\cdots\!\!$} \\
\multicolumn{7}{c}
  {$\!\!\!\cdots\cdots\cdots\cdots\cdots\cdots
  \cdots\cdots\cdots\cdots\cdots\cdots\!\!$}\\ \hline
    0&0&0  &$\cdots$ &$H_{{r-2}}$&$*$&$*$\\
    0&0&0  &$\cdots$ &0&$*$&$*$\\ \hline
    0&0&0  &$\cdots$ &0&$H_{{r-1}}$&$*$\\
    0&0&0  &$\cdots$ &0&0&$*$\\ \hline
    0&0&0  &$\cdots$ &0&0&$H_{r}$\\
    0&0&0  &$\cdots$ &0&0&0\\
\end{tabular}\right]},
\end{equation}
(where each $H_{i}$ is a nonsingular
$l_i\times l_i$ matrix) starting from
the lower strip, by unitary
row-transformations within horizontal
strips and by unitary
column-transformations. These
column-transformations are
transformations at the vertex $r+1$ and
they change $A_{r+1}$; denote the
obtained matrix by $A_{r+1}'$. Denote
also by ${\cal P}_{r}$ the set
consisting of the elements of ${\cal
P}_{r-1}$ and $k_i-l_i$ representations
of the form ${\mathbb L}_{p_ir}$ for
all $i=1,\dots,r$. Construct the quiver
representation\\[2em]
\[
\unitlength 0.4mm \linethickness{0.4pt}
\begin{picture}(25.00,25.00)(-45,7)
\put(13.00,38.00){\makebox(0,0)[lc]
{\fbox{$\scriptstyle{ A_{r+1}'}$}}}
\put(0.00,25.00){\makebox(0,0)[rc]
{${%\scriptstyle
 (l_{1} \text{\
copies})} \quad p_1
 \stackrel{I_{1}}{\lli}\cdots
 \stackrel{I_{1}}{\lli} (r+1)$}}
\multiput(-143.00,8.00)(5,0){29}{$\cdot$}
\put(0.00,-5.00){\makebox(0,0)[rc]
{${%\scriptstyle
 (l_{r}\text{ copies})}
\quad p_r
 \stackrel{I_{1}}{\lli}
 \cdots
 \stackrel{I_{1}}{\lli} (r+1)$}}
\put(-190.00,10.00){\makebox(0,0)[cc]
{${\mathbb M}_{r}$:}}
\put(25.00,10.00){\makebox(0,0)[lc]{$(r+2)
\stackrel{A_{r+2}}{\lllin}
  \cdots
  \stackrel{A_{t-1}}{\lllin}
  t$}}
\put(3.00,24.00){\line(5,-3){18.00}}
\put(3.00,-4.00){\line(5,3){18.00}}
\put(3.00,42.00){\line(3,-4){18.00}}
\put(0.00,47.00){\makebox(0,0)[rc]
{${%\scriptstyle
 (d_{r+1}-l_1-\dots-l_{r}\text{
copies)}} \quad \, (r+1)$}}
\end{picture}
\]\\

\subsection*{The result:}
After the step $t-1$, we have obtained
the set ${\cal P}_{t-1}$ consisting of
representations of the form ${\mathbb
L}_{ij}$, $j<t$, and the quiver
representation ${\mathbb M}_{t-1}$ (see
\eqref{x2.1da}), which may be
considered as a set of representations
of the form ${\mathbb L}_{it}$. Define
the representation
\begin{equation}\label{x2.1g}
  {\mathbb L}({\mathbb A})
  =\bigoplus_{{\mathbb
L}_{ij}\in{\cal P}_{t-1}\cup {\mathbb
M}_{t-1}}{\mathbb L}_{ij}.
\end{equation}

The following proposition will be
proved in Section \ref{xs2a}.

\begin{proposition} \label{t2.a}
The representation ${\mathbb
L}({\mathbb A})$ is the canonical form
$($see Theorem \ref{xt1.1}$)$ of a
matrix representation $\mathbb{A}$ of
the quiver \eqref{x1.5}.
\end{proposition}

\section{Cycles of linear mappings}
\label{s1.3}

In this section, we give an algorithm
for constructing a regularizing
decomposition \eqref{00} that involves
only unitary transformations. In the
same way, one may construct a
regularizing decomposition over an
arbitrary field using elementary
transformations.

By analogy with Section \ref{s1}, we
say that a matrix representation
$\mathbb A=(A_1,\dots, A_t)$ of a cycle
$\cal C$ (see \eqref{1.2}) is {\it
regular} if
\[ \dim_1(\mathbb A)=\dots
=\dim_t(\mathbb A)\] and all the
matrices $A_1,\dots, A_t$ are
nonsingular; otherwise the
representation is {\it singular}. A
decomposition
\begin{equation}\label{1.7b}
\mathbb A\simeq \mathbb D \oplus
\dots\oplus \mathbb G \oplus\mathbb{P}
\end{equation}
is a {\it regularizing decomposition}
of $\mathbb A$ if $\mathbb D, \dots,
\mathbb G$ are matrix representations
of the form ${\mathbb G}_{ij}$ (see
Lemma \ref{l00}) and $\mathbb P$ is a
regular representation.  By Theorem
\ref{t1.1}, the regularizing
decomposition \eqref{1.7b} is
determined uniquely up to isomorphism
of summands.

The algorithm works like a jack-plane
in a woodworker's hands. Starting from
the vertex $1$, we cut a {\it shave}:

$$
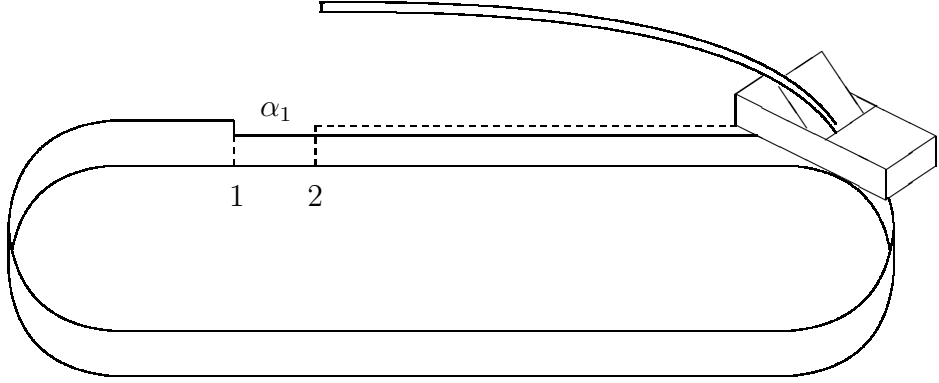
\unitlength 0.40mm
\linethickness{0.4pt}
\begin{picture}(353.47,164.76)(30,-18)
\put(80.00,30.00){\line(1,0){225.00}}
\put(305.00,45.00){\line(-1,0){225.00}}
\bezier{260}(80.00,30.00)(47.50,32.50)(45.00,65.00)
\bezier{260}(80.00,45.00)(47.50,47.50)(45.00,80.00)
\bezier{260}(80.00,115.00)(47.50,112.50)(45.00,80.00)
\put(45.00,65.00){\line(0,1){15.00}}
\bezier{220}(80.00,100.00)(51.00,98.00)(46.33,72.00)
\bezier{260}(304.33,30.00)(336.83,32.50)(339.33,65.00)
\bezier{260}(304.33,45.00)(336.83,47.50)(339.33,80.00)
\put(339.33,65.00){\line(0,1){15.00}}
\bezier{220}(304.33,100.00)(333.33,98.00)(338.00,72.00)
\put(80.00,100.00){\line(1,0){225.00}}
\put(80.00,115.00){\line(1,0){40.00}}
\put(120.00,115.00){\line(0,-1){5.00}}
\multiput(120.00,110.00)(0,-3.9){3}{\line(0,-1){2}}
\put(134,118){\makebox(0,0)[cc]{$\alpha_{1}$}}
\put(120.99,90){\makebox(0,0)[cc]{$1$}}
\put(120.00,110.00){\line(1,0){15.00}}
\put(135.00,110.00){\line(1,0){159.00}}
\multiput(146.99,113.23)(3.95,0){36}{\line(1,0){2}}
\multiput(146.99,113.23)(0,-3.7){4}{\line(0,-1){2}}
\put(146.99,90){\makebox(0,0)[cc]{$2$}}
\put(336.67,88.71){\line(3,2){16.50}}
\put(336.67,98.71){\line(3,2){16.50}}
\put(353.45,99.82){\line(0,1){10.00}}
\put(336.67,98.71){\line(0,-1){10.00}}
\put(336.67,88.71){\line(-2,1){50.00}}
\put(286.67,113.71){\line(0,1){10.00}}
\put(286.67,123.71){\line(2,-1){50.00}}
\put(316.67,108.71){\line(3,2){17.00}}
\put(315.45,138.15){\line(-3,-2){12.00}}
\put(315.45,138.15){\line(2,-3){13.78}}
\put(286.67,123.71){\line(3,2){9.00}}
\put(353.47,109.86){\line(-2,1){28.54}}
\put(301.03,126.50){\line(2,-3){9.42}}
\bezier{36}(339.26,80.00)(339.07,84.63)(337.59,89.26)
\bezier{776}(320.00,111.00)(290.48,152.85)(148.57,150.95)
\bezier{776}(320.00,113.81)(290.48,156.19)(148.57,154.28)
\put(149.00,150.95){\line(0,1){3.33}}
\end{picture}\\*[-17mm]
$$
We make a full circle by the jack-plane
and continue the process until the
shave breaks away. Then we transpose
all matrices of the remaining
representation and repeat this process.
The obtained representation
$\mathbb{P}$ of $\cal{C}$ is regular,
and the shaves split into a direct sum
of matrix representations of the form
${\mathbb{G}}_{ij}$.

Note that this proves Theorem
\ref{t1.1} since $\mathbb P$ is
isomorphic to a matrix representation
$(I_n,\dots,I_n,J)$, where $J$ is a
nonsingular Jordan (or Frobenius)
canonical matrix with respect to
similarity; see the end of Section
\ref{s1}. Hence $\mathbb A$ is
isomorphic to a direct sum of
representations of the form (i)--(iii)
from Theorem \ref{t1.1}. The uniqueness
of this decomposition follows from
Theorem \ref{t.0}.

\subsection*{The algorithm for cycles:}

This algorithm for every matrix
representation
\begin{equation}\label{2.1.1}
\unitlength 1.00mm
\linethickness{0.4pt}
\begin{picture}(85.33,5.00)
(10,0)
\put(25.33,0.00){\makebox(0,0)[cc]{$1$}}
\put(65.33,0.00){\makebox(0,0)[cc]{$\cdots$}}
\put(85.33,0.00){\makebox(0,0)[cc]{$t$}}
\put(45.33,0.00){\makebox(0,0)[cc]{$2$}}
\put(30.33,0.00){\line(1,0){10.00}}
\put(50.33,0.00){\line(1,0){10.00}}
\put(70.33,0.00){\line(1,0){10.00}}
\put(35.33,5.00){\makebox(0,0)[ct]{${A}_1$}}
\put(55.33,5.00){\makebox(0,0)[ct]{${A}_2$}}
\put(75.33,5.00){\makebox(0,0)[ct]{${A}_{t-1}$}}
\put(0.33,0.00){\makebox(0,0)[cc]{$\mathbb
A:$}}
\put(55.33,-7.33){\makebox(0,0)[cc]{$A_t$}}
%\bezvec{208}(81.00,-1.67)(55.33,-7.33)(29.83,-1.67)
\put(29.83,-1.67){\line(-4,1){0.2}}
\bezier{208}(81.00,-1.67)(55.33,-7.33)(29.83,-1.67)
%\end
\end{picture}\\*[20pt]
\end{equation}
of a cycle $\cal C$ (see \eqref{1.2})
constructs a decomposition
\begin{equation}\label{1.xe}
\mathbb{A}\simeq
\mathbb{P}(\mathbb{A}')\oplus
\widetilde{\mathbb{A}},
\end{equation}
where $\mathbb{A}'$ is formed by the
matrices of a chain of linear mappings,
$\mathbb{P}$ sends $\mathbb{A}'$ to a
representation of $\cal C$ that is
isomorphic to a direct sum of
representations of the form
${\mathbb{G}}_{ij}$ (see \eqref{ppp}
and compare with Example \ref{e1.a}),
and $\widetilde{\mathbb{A}}$ is a
representation of $\cal C$ that
satisfies the following condition for
each arrow:
\begin{equation}\label{awa}
\parbox{22em}
{If the arrow is oriented clockwise,
then the matrix\\ assigned to it has
linearly independent rows.}
\end{equation}
In steps $1,\, 2,\,\dots$ of the
algorithm we will construct quiver
representations ${\mathbb A}^{(1)}$,
${\mathbb A}^{(2)}$,\,\dots\,.

\medskip

\noindent {\bf Steps $\pmb
{1,2,\dots,l-1}$:} In step $1$ of the
algorithm, we check the condition
\eqref{awa} for the representation
$\mathbb A$ and the arrow $\alpha_1$.
If this condition holds, we put
${\mathbb A}^{(1)}={\mathbb A}$. If
this condition holds for $\alpha_2$
too, we put ${\mathbb A}^{(2)}={\mathbb
A}$, and so on.

If after $t$ steps we found that this
condition holds for all arrows of $\cal
C$, then we put
\begin{equation}\label{dd}
l=t+1,\qquad \mathbb{A}'=0, \qquad
\widetilde{\mathbb{A}}=\mathbb{A}
\end{equation}
and stop the algorithm. Otherwise, we
set
\begin{equation}\label{mmm}
 l=\min\Bigl\{i\in\{1,\dots,t\}\,\Bigl|\,
 \alpha_i \text{ in $\mathbb A$
  does not satisfy
 \eqref{awa}}\Bigr\}
\end{equation}
and continue the algorithm as follows:
\medskip

\noindent {\bf Step {\textit l}:}  By
unitary transformations at the vertex
$[l+1]$, we reduce the matrix $A_l$ of
$\mathbb A$ to a matrix
\begin{equation*}\label{2.1.5}
   {\left[\begin{tabular}{c}
    $0$ \\ \hline
    \raisebox{-5pt}{$A_{l}^{(l)}$}
\end{tabular}\right]}\!\!\!\!\!
\begin{tabular}{l}
    $\,\}
    {d_{(l+1)'}^{(l)}\:
    \text{rows}}$\\[1mm]
    $\,\}%\scriptstyle
    {d_{[l+1]}^{(l)}\: \text{rows}}$\\[2mm]
\end{tabular},\quad
d_{(l+1)'}^{(l)}>0,
\end{equation*}
where the rows of $A_{l}^{(l)}$ are
linearly independent. This changes
$A_{l+1}$; we denote the obtained
matrix by $A^{(l)}$ and construct the
representation\\[-10mm]
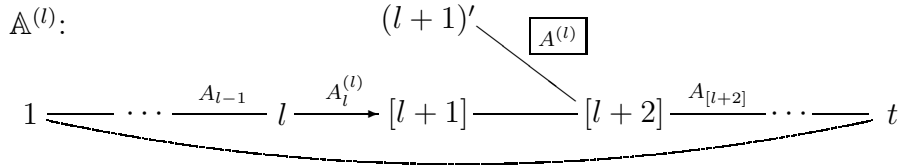
$$
\unitlength 1.25mm
\linethickness{0.4pt}
\begin{picture}(90.33,28.00)(0,-8)
\put(0.00,10.00){\makebox(0,0)[cc]{${\mathbb
A}^{(l)}$:}}
\bezier{348}(88.67,-0.67)(45.00,-10.00)(1.00,-0.67)
\put(41.50,10.00){\makebox(0,0)[cc]{$(l+1)'$}}
\put(0.00,0.00){\makebox(0,0)[rc]{$1$}}
\put(1.00,0.00){\line(1,0){7.00}}
\put(11.60,0.00){\makebox(0,0)[cc]{$\cdots$}}
\put(27.26,0.00){\vector(1,0){9.00}}
\put(32.66,2.50){\makebox(0,0)[cc]
{$\scriptstyle{A_l^{(l)}}$}}
\put(41.50,0.00){\makebox(0,0)[cc]{$[l+1]$}}
\put(46.33,0.00){\line(1,0){11.00}}
\put(62.33,0.00){\makebox(0,0)[cc]{$[l+2]$}}
\put(26.,0.00){\makebox(0,0)[cc]{$l$}}
\put(67.33,0.00){\line(1,0){10.00}}
\put(72.33,2.00){\makebox(0,0)[cc]
{$\scriptstyle{A_{[l+2]}}$}}
\put(80.1,0.00){\makebox(0,0)[cc]{$\cdots$}}
\put(82.33,0.00){\line(1,0){6.00}}
\put(90.33,0.00){\makebox(0,0)[lc]{$t$}}
\put(46.66,9.33){\line(4,-3){10.50}}
\put(55.33,8.33){\makebox(0,0)[cc]
{\fbox{$\scriptstyle{A^{(l)}}$}}}
\put(14.33,0.00){\line(1,0){10.00}}
\put(19.66,2.00){\makebox(0,0)[cc]
{$\scriptstyle{A_{l-1}}$}}
\end{picture}
%\\[2.5em]
$$
(the other matrices are the same as in
\eqref{2.1.1}). Its dimensions at the
vertices $(l+1)'$ and $[l+1]$ are
$d_{(l+1)'}^{(l)}$ and
$d_{[l+1]}^{(l)}$, and the arrow
$(l+1)'\llin [l+2]$ has the orientation
of $[l+1]\llin [l+2]$. The matrix
$A^{(l)}$ is partitioned into the
strips $ A^{(l)}_{(l+1)'}$ and
$A^{(l)}_{[l+1]}$, which are assigned
to the arrows $(l+1)'\llin [l+2]$ and $
[l+1]\llin [l+2]$, see \eqref{x2.1acc}.
\medskip

\noindent{\textbf{Step $\pmb{r\ (r>
l)}$:}} Assume we have constructed in
step $r-1$ a representation
\\[1em]
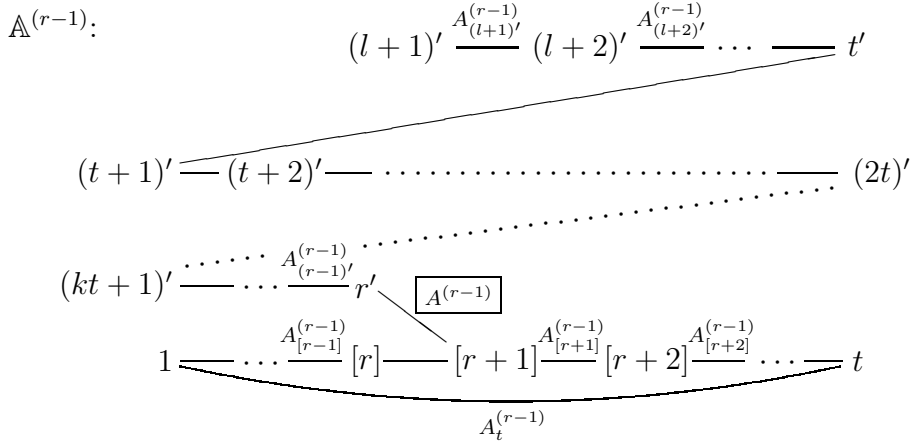
\begin{equation*}\label{m}
\unitlength 1mm \linethickness{0.4pt}
\begin{picture}(90.33,22.33)(0,22)
\put(-10.00,45.00){\makebox(0,0)[rc]
{${\mathbb A}^{(r-1)}$:}}
\put(0.00,25.00){\makebox(0,0)[rc]
{$(t+1)'$}}
\put(1.00,25.00){\line(1,0){5.00}}
\put(13.50,25.00){\makebox(0,0)[cc]
{$(t+2)'$}}
\put(20.33,25.00){\line(1,0){6.0}}
\multiput(28.,24.0)(2,0){26}{$\cdot$}
\put(80.33,25.00){\line(1,0){8.00}}
\put(90.33,25.00){\makebox(0,0)[lc]
{$(2t)'$}}
\bezier{348}(88.67,-0.67)(45.00,-10.00)(1.00,-0.67)\put(0.00,10.00){\makebox(0,0)[rc]{$(kt+1)'$}}
\put(1.00,10.00){\line(1,0){7.00}}
\put(12.00,9.70){\makebox(0,0)[cc]
{$\cdots$}}
\put(15.33,10.00){\line(1,0){8.00}}
\put(19.33,13.3){\makebox(0,0)[cc]
{$\scriptstyle{A_{(r-1)'}^{(r-1)}}$}}
\put(25.67,10.00){\makebox(0,0)[cc]{$r'$}}
\multiput(1.33,11.67)(2,0.25){6}{$\cdot$}
\multiput(25,14.3)(2,0.25){32}{$\cdot$}
\put(0.00,0.00){\makebox(0,0)[rc]{$1$}}
\put(1.00,0.00){\line(1,0){7.00}}
\put(12.00,-0.30){\makebox(0,0)[cc]
{$\cdots$}}
\put(15.33,0.00){\line(1,0){7.0}}
\put(18.7,3.30){\makebox(0,0)[cc]
{$\scriptstyle{A_{[r-1]}^{(r-1)}}$}}
\put(25.67,0.00){\makebox(0,0)[cc]{$[r]$}}
\put(28.00,0.00){\line(1,0){8.50}}
\put(43.00,0.00){\makebox(0,0)[cc]{$[r+1]$}}
\put(49.00,0.00){\line(1,0){7.20}}
\put(53.00,3.30){\makebox(0,0)[cc]
{$\scriptstyle{A^{(r-1)}_{[r+1]}}$}}
\put(63.00,0.00){\makebox(0,0)[cc]{$[r+2]$}}
\put(69.00,0.00){\line(1,0){7.00}}
\put(73.00,3.30){\makebox(0,0)[cc]
{$\scriptstyle{A_{[r+2]}^{(r-1)}}$}}
\put(80.7,-0.20){\makebox(0,0)[cc]{$\cdots$}}
\put(84,0.00){\line(1,0){5.00}}
\put(90.33,0.00){\makebox(0,0)[lc]{$t$}}
\put(27.33,9.33){\line(4,-3){9.00}}
%\stackrel{A_{t-1}}{\lllin}
\put(92.2,43.7){\makebox(0,0)[rc]
{$(l+1)'
\stackrel{A^{(r-1)}_{(l+1)'}}{\lllin}
(l+2)'
\stackrel{A^{(r-1)}_{(l+2)'}}{\lllin}
\cdots\
  \frac{}{\qquad}\ t'$}}
\put(1.00,26.33){\line(6,1){87.00}}
\put(38.00,8.83){\makebox(0,0)[cc]
{\fbox{$\scriptstyle{A^{(r-1)}}$}}}
\put(45.16,-8.5){\makebox(0,0)[cc]
{$\scriptstyle{A_{t}^{(r-1)}}$}}
\end{picture}\\*[8.6em]
\end{equation*}
where each arrow $\alpha_{i'}:i'\llin
(i+1)'$ has the orientation of
$\alpha_{[i]}:[i]\llin [i+1]$ in $\cal
C$, and $\alpha_{r'}:r'\llin [r+1]$ has
the orientation of
$\alpha_{[r]}:[r]\llin [r+1]$.

We will reduce ${\mathbb A}^{(r-1)}$ by
unitary transformations at the vertex
$[r+1]$:
\begin{itemize}
  \item[(i)]
If $\alpha_{[r]}$ is oriented
clockwise, then $A^{(r-1)}$ consists of
two vertical strips with
$\dim_{r'}{\mathbb A}^{(r-1)}$ and
$\dim_{[r]}{\mathbb A}^{(r-1)}$ columns
(see \eqref{x2.1aac}--\eqref{2.1.5aq});
we reduce it by unitary
row-transformations as follows:
\begin{equation} \label{2.3''a}
A^{(r-1)}=[A_{r'}^{(r-1)}|A_{[r]}^{(r-1)}]
\mapsto \left[\begin{tabular}{c|c}
   $A_{r'}^{(r)}$ & 0 \\
   $*$ & $A_{[r]}^{(r)}$
\end{tabular}
\right],
\end{equation}
where $A_{[r]}^{(r)}$ has linearly
independent rows.

  \item[(ii)]
If $\alpha_{[r]}$ is oriented
counterclockwise, then $A^{(r-1)}$
consists of two horizontal strips with
$\dim_{r'}{\mathbb A}^{(r-1)}$ and
$\dim_{[r]}{\mathbb A}^{(r-1)}$ rows;
we reduce it by unitary
column-transformations as follows:
\begin{equation} \label{2.3'''}
A^{(r-1)}=\left[\begin{tabular}{c}
  $A_{r'}^{(r-1)}$ \\ \hline
  \raisebox{-1.5mm}{$A_{[r]}^{(r-1)}$}
\end{tabular} \right]\mapsto
\left[\begin{tabular}{cc}
   $A_{r'}^{(r)}$ & 0 \\ \hline
   $*$ &
   \raisebox{-1.1mm}{$A_{[r]}^{(r)}$}
\end{tabular}\right],
\end{equation}
where $A_{r'}^{(r)}$ has linearly
independent columns.
\end{itemize}
These unitary transformations at the
vertex $[r+1]$ change the matrix
$A_{[r+1]}^{(r-1)}$ too; we denote the
obtained matrix by $A^{(r)}$ and
construct the representation
\\[1.5em]
\begin{equation}\label{1ww}
\unitlength 1.1mm \linethickness{0.4pt}
\begin{picture}(90.33,24.33)(1.5,20)
\put(-3.50,45.00){\makebox(0,0)[rc]
{${\mathbb A}^{(r)}$:}}
\put(0.00,25.00){\makebox(0,0)[rc]{$(t+1)'$}}
\put(1.00,25.00){\line(1,0){6.00}}
\put(13.50,25.00){\makebox(0,0)[cc]{$(t+2)'$}}
\put(19.33,25.00){\line(1,0){7.00}}
\multiput(28.5,23.9)(1.96,0){26}{$\cdot$}
\put(80.33,25.00){\line(1,0){8.00}}
\put(90.33,25.00){\makebox(0,0)[lc]{$(2t)'$}}
\bezier{348}(88.67,-0.67)(45.00,-10.00)(1.00,-0.67)
\put(0.00,10.00){\makebox(0,0)[rc]{$(kt+1)'$}}
\put(1.00,10.00){\line(1,0){7.00}}
\put(12.00,9.7){\makebox(0,0)[cc]{$\cdots$}}
\put(15.33,10.00){\line(1,0){8.00}}
\put(19.33,12.70){\makebox(0,0)[cc]
{$\scriptstyle{A_{(r-1)'}^{(r-1)}}$}}
\put(25.67,10.00){\makebox(0,0)[cc]{$r'$}}
\multiput(1.33,13)(2.01,0.23){44}{$\cdot$}
\put(0.00,0.00){\makebox(0,0)[rc]{$1$}}
\put(1.00,0.00){\line(1,0){7.00}}
\put(12.04,-0.250){\makebox(0,0)[cc]{$\cdots$}}
\put(15.33,0.00){\line(1,0){8.00}}
\put(19.33,3){\makebox(0,0)[cc]
{$\scriptstyle{A_{[r-1]}^{(r-1)}}$}}
\put(33.33,3.0){\makebox(0,0)[cc]
{$\scriptstyle{A_{[r]}^{(r)}}$}}
\put(33.33,13){\makebox(0,0)[cc]
{$\scriptstyle{A_{r'}^{(r)}}$}}
\put(25.67,0.00){\makebox(0,0)[cc]{$[r]$}}
\put(28.00,0.00){\line(1,0){9.00}}
\put(43.00,0.00){\makebox(0,0)[cc]{$[r+1]$}}
\put(48.00,0.00){\line(1,0){9.00}}
\put(63.00,0.00){\makebox(0,0)[cc]{$[r+2]$}}
\put(68.00,0.00){\line(1,0){9.00}}
\put(73.00,3.0){\makebox(0,0)[cc]
{$\scriptstyle{A_{[r+2]}^{(r-1)}}$}}
\put(80.7,-0.25){\makebox(0,0)[cc]{$\cdots$}}
\put(84,0.00){\line(1,0){4.30}}
\put(90.33,0.00){\makebox(0,0)[lc]{$t$}}
\put(1.00,26.33){\line(6,1){87.00}}
\put(58.5,8.33){\makebox(0,0)[cc]
{\fbox{$\scriptstyle{A^{(r)}}$}}}
\put(28.00,10.00){\line(1,0){9.00}}
\put(43.00,10.00){\makebox(0,0)[cc]{$(r+1)'$}}
\put(48.00,8.33){\line(3,-2){9}}
\put(92,43.8){\makebox(0,0)[rc]
{$(l+1)'
\stackrel{A^{(r-1)}_{(l+1)'}}{\lllin}(l+2)'
\stackrel{A^{(r-1)}_{(l+2)'}}{\lllin}
\cdots\
  \frac{}{\qquad}\ t'$}}
\end{picture}\\*[33mm]
\end{equation}
where ${A}^{(r)}$ is partitioned into
two strips:
\begin{equation}\label{vvv}
 A^{(r)}=\begin{cases}
 \left[\begin{tabular}{c|c}
$ A^{(r)}_{(r+1)'}$&
 $A^{(r)}_{[r+1]}$
 \end{tabular}\right]&
\text{if $\alpha_{[r+1]}$ is oriented
clockwise}, \\[1em]
 \left[\begin{tabular}{c}
   $ A^{(r)}_{(r+1)'}$\\ \hline
    $A^{(r)}_{[r+1]}$
\end{tabular}\right] &
\text{if $\alpha_{[r+1]}$ is oriented
counterclockwise},
\end{cases}
\end{equation}
and these strips are assigned to the
arrows \[ (r+1)'\llin [r+2],\qquad
[r+1]\llin [r+2].\]

\subsection*{The result:}

We make at least $t$ steps and stop at
the first representation $\mathbb
A^{(n)}$ with
\begin{equation}\label{wwv}
n\ge t\qquad\text{and}\qquad
{A}^{(n)}_{(n+1)'}=0.
\end{equation}

The matrix ${A}^{(n)}_{(n+1)'}$ is
assigned to the arrow $(n+1)'\llin
[n+2]$. Deleting this arrow, we break
$\mathbb A^{(n)}$ into two
representations:\\[-3mm]
\begin{equation}\label{1k}
\unitlength 0.9mm \linethickness{0.4pt}
\begin{picture}(91.33,28.00)(-7,25)
\put(-30.00,40.00){\makebox(0,0)[cc]
{${\mathbb A}'$:}}
\put(0.00,25.00){\makebox(0,0)[rc]{$(t+1)'$}}
\put(1.00,25.00){\line(1,0){7.00}}
\put(16.50,25.00){\makebox(0,0)[cc]{$(t+2)'$}}
\put(23.33,25.00){\line(1,0){8.00}}
\multiput(34.5,23.6)(2,0){22}{$\cdot$}
\put(80.33,25.00){\line(1,0){8.00}}
\put(90.33,25.00){\makebox(0,0)[lc]{$(2t)'$}}
\put(-17.50,12.00){\makebox(0,0)[lc]
{$(kt+1)'\ \frac{}{\qquad\quad} \
\cdots\ \stackrel{
A^{(n)}_{n'}}{\lllin\!\!\!\llin}
(n+1)'$}}
\multiput(1.33,13)(2,0.22){44}{$\cdot$}
\put(91.33,44.50){\makebox(0,0)[rc]
{$(l+1)'\stackrel{A^{(n)}_{(l+1)'}}{\lllin}
(l+2)'\stackrel{A^{(n)}_{(l+2)'}}{\lllin}
\cdots\, \lllin\ t'$}}
\put(1.00,26.33){\line(6,1){85.70}}
\end{picture}
\\*[5em]
\end{equation}
and\\
\begin{equation}\label{1h}
\unitlength 1.15mm
\linethickness{0.4pt}
\begin{picture}(90.33,6.00)(0,-1)
\put(-9.00,1.00){\makebox(0,0)[cc]
{$\widetilde{\mathbb{A}}$:}}
\bezier{348}(89.67,-0.67)(45.00,-10.00)(1.00,-0.67)
\put(0.00,0.00){\makebox(0,0)[rc]{$1$}}
\put(1.00,0.00){\line(1,0){7.00}}
\put(12.00,-0.15){\makebox(0,0)[cc]{$\cdots$}}
\put(15.33,0.00){\line(1,0){8.00}}
\put(33.33,3){\makebox(0,0)[cc]
{$\scriptstyle{A_{[n]}^{(n)}}$}}
\put(25.67,0.00){\makebox(0,0)[cc]{$[n]$}}
\put(28.00,0.00){\line(1,0){9.50}}
\put(43.00,0.00){\makebox(0,0)[cc]{$[n+1]$}}
\put(48.00,0.00){\line(1,0){9.50}}
\put(53.00,3){\makebox(0,0)[cc]
{$\scriptstyle{A^{(n)}_{[n+1]}}$}}
\put(63.00,0.00){\makebox(0,0)[cc]{$[n+2]$}}
\put(71.00,0.00){\line(-1,0){3.00}}
\put(68.00,0.00){\line(1,0){10.00}}
\put(73.00,3){\makebox(0,0)[cc]
{$\scriptstyle{A_{[n+2]}^{(n)}}$}}
\put(81.9,-0.15){\makebox(0,0)[cc]{$\cdots$}}
\put(84.83,0.00){\line(1,0){5.00}}
\put(91.33,0.00){\makebox(0,0)[lc]{$t$}}
\end{picture}
\\*[2em]
\end{equation}
The representation ${\mathbb A}'$ is a
representation of the quiver
\begin{equation}\label{2.3v}
  (l+1)'\:\frac {\alpha_{(l+1)'}}
  {\qquad\qquad}\:(l+2)'\:
  \frac{\alpha_{(l+2)'}}{\qquad\qquad}\:
  \cdots\:
  \frac{\alpha_{n'}}{\quad\qquad}\:
  (n+1)',
\end{equation}
whose arrows $i'\llin (i+1)'$ have the
orientation of the arrows
$\alpha_{[i]}:[i]\llin [i+1]$ in $\cal
C$. By analogy with Example \ref{e1.a},
we construct the mapping $\mathbb P$
that sends a representation $\mathbb B$
of the quiver \eqref{2.3v} to a
representation $\mathbb D$ of the cycle
$\cal C$:\\[4mm]
\begin{equation}\label{2.3w}
\unitlength 0.8mm \linethickness{0.4pt}
\begin{picture}(91.33,30.00)(-8,13)
\put(0.00,25.00){\makebox(0,0)[rc]{
$\scriptstyle{(t+1)'}$}}
\put(1.00,25.00){\line(1,0){11.00}}
\put(20,25.00){\makebox(0,0)[cc]
{$\scriptstyle{(t+2)'}$}}
\put(27,25.00){\line(1,0){10.00}}
\multiput(39,23.6)(2.4,0){16}{$\cdot$}
\put(78.33,25.00){\line(1,0){10.00}}
\put(90.33,25.00){\makebox(0,0)[lc]
{$\scriptstyle{(2t)'}$}}
\bezier{348}(88.67,-12.00)(45.00,-21.33)(1.00,-12.00)
\put(1.00,26.33){\line(6,1){88.00}}
\put(0.00,-11.33){\makebox(0,0)[rc]{
$\scriptstyle{1}$}}
\put(1.00,-11.33){\line(1,0){15.00}}
\put(20,-11.33){\makebox(0,0)[cc]
{$\scriptstyle{2}$}}
\put(23,-11.33){\line(1,0){14.00}}
\multiput(39,-12.73)(2.4,0){16}{$\cdot$}
\put(78.33,-11.33){\line(1,0){10.00}}
\put(90.33,-11.33){\makebox(0,0)[lc]
{$\scriptstyle{t}$}}
\bezier{348}(88.67,-12.00)(45.00,-21.33)(1.00,-12.00)
\put(1.00,26.33){\line(6,1){88.00}}
\put(0.00,10.00){\makebox(0,0)[rc]
{$\scriptstyle{(kt+1)'}$}}
\put(1.00,10.00){\line(1,0){11.00}}
\put(20.00,10.00){\makebox(0,0)[cc]
{$\scriptstyle{(kt+2)'}$}}
\put(27,10.00){\line(1,0){10.00}}
\multiput(1.33,11)(2.46,0.32){36}{$\cdot$}
\put(47.50,10.00){\line(1,0){10.00}}
\put(43.00,9.50){\makebox(0,0)[cc]
{$\cdots$}}
\put(64.00,10.00){\makebox(0,0)[cc]
{$\scriptstyle{(n+1)'}$}}
\put(93.53,43.5){\makebox(0,0)[rc]
{${\scriptstyle (l+1)'}\!\!
\stackrel{B_{(l+1)'}}{\lllin}
{\scriptstyle (l+2)'}
\stackrel{B_{(l+2)'}}{\lllin}\cdots
\stackrel{B_{(t-1)'}}{\lllin}
{\scriptstyle t'}$}}
\put(-30,24.67){\makebox(0,0)[cc]
{${\mathbb B}:$}}
\put(-30,-11.33){\makebox(0,0)[cc]
{${\mathbb D}:$}}
\put(-35,8){\makebox(0,0)[cc]
{${\mathbb P}$}}
\put(-31.5,21.67){\vector(0,-1){29.00}}
\put(-33,21.67){\line(1,0){3.00}}
\put(45.00,36.00){\makebox(0,0)[cc]
{$\scriptstyle{B_{t'}}$}}
\put(45.00,-20.00){\makebox(0,0)[cc]
{$\scriptstyle{D_{t}}$}}
\put(33,28){\makebox(0,0)[cc]
{$\scriptstyle{B_{(t+2)'}}$}}
\put(33,12){\makebox(0,0)[cc]
{$\scriptstyle{B_{(kt+2)'}}$}}
\put(53,12){\makebox(0,0)[cc]
{$\scriptstyle{B_{n'}}$}}
\put(33,-9){\makebox(0,0)[cc]
{$\scriptstyle{D_{2}}$}}
\put(83,-9){\makebox(0,0)[cc]
{$\scriptstyle{D_{t-1}}$}}
\put(83,28){\makebox(0,0)[cc]
{$\scriptstyle{B_{(2t-1)'}}$}}
\put(14.00,22.00){\makebox(0,0)[rc]
{$\scriptstyle{B_{(t+1)'}}$}}
\put(14.00,7.00){\makebox(0,0)[rc]
{$\scriptstyle{B_{(kt+1)'}}$}}
\put(12.00,-9.00){\makebox(0,0)[rc]
{$\scriptstyle{D_{1}}$}}
\end{picture}\\*[7em]
\end{equation}
This mapping is known in representation
theory as a {\it push-down functor}
(see \cite[Sect. 14.3]{gab+roi}) and is
determined as follows:
\begin{equation}\label{eee}
D_i=\bigoplus_{\substack{[j]=i\\ l\le
j\le n+1}} B_{j'},\qquad i=1,2,\dots,
t,
\end{equation}
(i.e., $D_i$ is the direct sum of all
$B_{j'}$ disposed over it), where
\begin{equation}\label{jjj}
B_{l'}=0_{p0} \quad\text{with}\quad
  p=\dim_{(l+1)'}{\mathbb
  B}
\end{equation}
(recall that the arrow $\alpha_{l}$ is
oriented clockwise, see step $l$ of the
algorithm), and
 \[
 B_{(n+1)'}= \begin{cases}
    0_{0q} & \text{if $\alpha_{[n+1]}:[n+1]
    \longrightarrow [n+2]$}, \\
    0_{q0} & \text{if $\alpha_{[n+1]}:[n+1]
    \longleftarrow [n+2]$},
 \end{cases}
    \quad\text{with
  $q=\dim_{(n+1)'}{\mathbb
  B}$}.
\]

(The definition of ${\mathbb P}:
{\mathbb B}\mapsto {\mathbb D}$ becomes
clearer if the representations
${\mathbb B}$ and ${\mathbb D}$ are
given by vector spaces and linear
mappings: each vector space of
${\mathbb D}$ is the direct sum of the
vector spaces of ${\mathbb B}$ disposed
over it, and each linear mapping of
${\mathbb D}$ is determined by the
linear mappings of ${\mathbb B}$
disposed over it.)

The following proposition will be
proved in Section \ref{subs}.

\begin{proposition}\label{l2}
Let the algorithm for circles transform
a matrix representation $\mathbb A$ of
a cycle $\cal C$ to $\mathbb{A}'$ and
$\widetilde{\mathbb{A}}$. Then
\begin{itemize}
  \item[{\rm (a)}]
The condition \eqref{awa} holds for
$\widetilde{\mathbb{A}}$ and all
arrows.
  \item[{\rm (b)}]
If an arrow $\alpha_i$ is oriented
counterclockwise and the columns of
${A_i}$ are linearly independent, then
the columns of $\widetilde{A}_i$ are
linearly independent too.
  \item[{\rm (c)}] $\mathbb{A}\simeq
\mathbb{P}(\mathbb{A}')\oplus
\widetilde{\mathbb{A}}$.
\end{itemize}
\end{proposition}

\section{Main theorem}\label{s_main}

\begin{theorem}\label{th}
A regularizing decomposition
\eqref{1.7b} of a matrix representation
$\mathbb A$ of a cycle $\cal C$ can be
constructed in 3 steps using only
unitary transformations:
\begin{itemize}
  \item[$1.$]
Applying the algorithm for cycles to
${\mathbb A}$, we get $\mathbb{A}\simeq
\mathbb{P}(\mathbb{A}')\oplus
\widetilde{\mathbb{A}}$.
  \item[$2.$]
Applying the algorithm for cycles to
the matrix representation ${\mathbb
B}:=\widetilde{\mathbb{A}}^T$ of the
cycle ${\cal C}^{\,T}$ $($see
\eqref{1.xy}$)$, we get
$\widetilde{\mathbb{A}}^T\simeq
\mathbb{P}(\mathbb{B}')\oplus
\widetilde{\mathbb{B}}$.
  \item[$3.$]
Applying the algorithm for chains to
$\mathbb{A}'$ and $\mathbb{B}^{\prime
T}$, we get
\begin{equation}\label{pp}
  \mathbb{A}'\simeq {\mathbb
L}_{i_1j_1}\oplus \dots\oplus {\mathbb
L}_{i_pj_p},\quad\mathbb{B}^{\prime
T}\simeq {\mathbb
L}_{i_{p+1}j_{p+1}}\oplus \dots\oplus
{\mathbb L}_{i_qj_q}.
\end{equation}
\end{itemize}
Then
\begin{equation}\label{qq}
\mathbb{A}\simeq {\mathbb
G}_{i_1j_1}\oplus \dots\oplus {\mathbb
G}_{i_qj_q}\oplus\widetilde{\mathbb{B}}^T
\end{equation}
$($see \eqref{ppp}$)$ and the
representation
$\widetilde{\mathbb{B}}^T$ is regular.
\end{theorem}

\begin{proof}
By \eqref{1.yx} and Proposition
\ref{l2}(c),
\[
\mathbb{A}\simeq
\mathbb{P}(\mathbb{A}') \oplus
(\mathbb{P}(\mathbb{B}'))^T\oplus
\widetilde{\mathbb{B}}^T=
\mathbb{P}(\mathbb{A}') \oplus
\mathbb{P}(\mathbb{B}'^T)\oplus
\widetilde{\mathbb{B}}^T.
\]
Substituting \eqref{pp}, we obtain
\[
\mathbb{A}\simeq \mathbb{P}({\mathbb
L}_{i_1j_1})\oplus \dots\oplus
\mathbb{P}({\mathbb L}_{i_qj_q})
\oplus\widetilde{\mathbb{B}}^T.
\]
This
proves \eqref{qq} since
$\mathbb{P}({\mathbb L}_{ij})= {\mathbb
G}_{ij}$.

Let us prove that
$\widetilde{\mathbb{B}}^T$ is regular.
By Proposition \ref{l2}(a), every
matrix of $\widetilde{\mathbb{A}}$
assigned to an arrow oriented clockwise
has linearly independent rows. The
matrix representation ${\mathbb{B}}=
\widetilde{\mathbb{A}}^T$ is
constructed by transposing all
matrices, and it is a representation of
the cycle ${\cal C}^T$ obtained from
${\cal C}$ by changing the direction of
each arrow. Hence every matrix of
${\mathbb{B}}$ assigned to an arrow
oriented counterclockwise has linearly
independent columns; by Proposition
\ref{l2}(b) the same holds for the
matrices of $\widetilde{\mathbb{B}}$.
Moreover, by Proposition \ref{l2}(a)
every matrix of
$\widetilde{\mathbb{B}}$ assigned to an
arrow oriented clockwise has linearly
independent rows. Hence,
\[
\dim_{[i+1]}\widetilde{\mathbb{B}}
=\rank \widetilde{{B}}_i \le
\dim_i\widetilde{\mathbb{B}}
\]
for all vertices $i=1,\dots,t$. We have
\[
\dim_1\widetilde{\mathbb{B}}\ge
\dim_2\widetilde{\mathbb{B}}\ge\cdots\ge
\dim_t\widetilde{\mathbb{B}}\ge
\dim_1\widetilde{\mathbb{B}}.
\]
Therefore, each matrix
$\widetilde{{B}}_i$ is square and its
rows or columns are linearly
independent. So $\widetilde{{B}}_i$ is
nonsingular and the representation
$\widetilde{\mathbb{B}}$ is regular.
Then $\widetilde{\mathbb{B}}^T$ is
regular too.
\end{proof}

\section{Proof of Proposition \ref{t2.a}}
 \label{xs2a}

In each step $r\in\{1,\dots,t-1\}$ of
the algorithm for chains (Section
\ref{s4}) we constructed the matrix
$B_r$ of the form \eqref{x2.6'} or
\eqref{x2.7'}. Denote by $D_r$ the
matrix obtained from $B_r$ by
replacement of all blocks $H_{i}$ by
$I_{l_i}$ and all blocks $*$ by $0$.
Let us prove that the representation
\begin{equation*}\label{x3.z}
{\mathbb D}_{r} :\qquad 1\ \frac
{D_1}{\qquad}\ \cdots\
  \frac{D_{r}}{\qquad}\
  (r+1)\ \frac{A_{r+1}'}{\qquad}\ (r+2)\
  \frac{A_{r+2}}{\qquad}\
  \cdots\
  \frac{{{A}}_{t-1}}{\qquad}\
  t
\end{equation*}
of the quiver \eqref{x1.5} is
isomorphic to the initial
representation $\mathbb A$:
\begin{equation}\label{x3.y}
  {\mathbb A}\simeq {\mathbb
  D}_{r},\qquad r=1,\dots,t-1.
\end{equation}

In step 1 we reduced $\mathbb A$ to
\[
{\mathbb B}_{1} :\qquad 1\ \frac
{B_1}{\qquad}\
  2\ \frac{A_{2}'}{\qquad}\ 3\
  \frac{A_{3}}{\qquad}\
  \cdots\
  \frac{{{A}}_{t-1}}{\qquad}\
  t
\]
by unitary transformations at vertices
$1$ and $2$ (see \eqref{x2.1'}). Using
transformations at vertex 1, we reduce
$B_1$ to
\begin{equation}\label{x2.ee}
D_1=
\begin{bmatrix}
0&I_k\\0&0
\end{bmatrix},
\end{equation}
and so ${\mathbb A}$ is isomorphic to
\[
{\mathbb D}_{1} :\qquad 1\ \frac
{D_1}{\qquad}\
  2\ \frac{A_{2}'}{\qquad}\ 3\
  \frac{A_{3}}{\qquad}\
  \cdots\
  \frac{{{A}}_{t-1}}{\qquad}\
  t.
\]
We may produce at vertex $2$ of
${\mathbb D}_1$ every transformation
given by a nonsingular block-triangular
matrix
\[
S_2=
\begin{bmatrix}
S_{11}&S_{12}\\0&S_{22}
\end{bmatrix},
\]
where $S_{11}$ is $k$-by-$k$ if
$\alpha_1:1\longrightarrow 2$, and
$S_{22}$ is $k$-by-$k$ if
$\alpha_1:1\longleftarrow 2$. This
transformation spoils the block $I_k$
of $D_1$ but we recover it by
transformations at vertex $1$.

Reasoning by induction on $r$, we
assume that ${\mathbb A}$ is isomorphic
to
\[ {\mathbb D}_{r-1} :\qquad 1\
   \frac {D_1}{\qquad}\ \cdots\
  \frac{D_{r-1}}{\qquad}\
  r\ \frac{A_{r}'}{\qquad}\ (r+1)\
  \frac{A_{r+1}}{\qquad}\
  \cdots\
  \frac{{{A}}_{t-1}}{\qquad}\ t
\]
(where $D_1,\dots,D_{r-1}$ are obtained
from $B_1,\dots,B_{r-1}$ and $A'_r$ is
taken from \eqref{x2.1c}) and that
transformations at vertices
$1,\dots,r-1$ may recover the matrices
$D_1,\dots,D_{r-1}$ of ${\mathbb
D}_{r-1}$ after each transformation at
vertex $r$ given by a nonsingular
block-triangular matrix
\begin{equation}\label{x3.x}
S_{r}= \begin{bmatrix}
S_{11}&S_{12}&\cdots&S_{1r}\\
&S_{22}&\cdots&S_{2r}\\&&\ddots&\vdots\\
0&&&S_{rr}
\end{bmatrix},
\end{equation}
in which the sizes of diagonal blocks
coincide with the sizes of horizontal
strips of $B_{r-1}$ if
$\alpha_{r-1}:(r-1)\longrightarrow r$,
or with the sizes of vertical strips of
$B_{r-1}$ if
$\alpha_{r-1}:(r-1)\longleftarrow r$,
see \eqref{x2.6'} and \eqref{x2.7'}.

In step $r$ of the algorithm, we
reduced $A_r'$ to $B_r$ of the form
\eqref{x2.6'} or \eqref{x2.7'} by
unitary transformations at the vertices
$r$ and $r+1$; moreover, we used only
those transformation at vertex $r$ that
were given by unitary block-diagonal
matrices partitioned as \eqref{x3.x}.
By the same transformations at the
vertices $r$ and $r+1$ of ${\mathbb
D}_{r-1}$, we reduce its matrix $A_r'$
to $B_r$. Then we reduce $B_r$ to $D_r$
by a transformation at vertex $r$ given
by a matrix of the form \eqref{x3.x},
and restore $D_1,\dots,D_{r-1}$ by
transformations at vertices
$1,\dots,r-1$. The obtained
representation is ${\mathbb D}_{r}$,
and so
\[
{\mathbb A}\simeq{\mathbb
D}_{r-1}\simeq{\mathbb D}_{r}.
\]
Moreover, we may produce at the vertex
$r+1$ of ${\mathbb D}_{r}$ all
transformations given by
block-triangular matrices, restoring
the matrix $D_r$ by transformations at
vertex $r$ given by matrices of the
form \eqref{x3.x}, and then restoring
$D_1,\dots,D_{r-1}$ by transformations
at the vertices $1,\dots,r-1$. This
proves the isomorphism \eqref{x3.y}.
\medskip

We now transform the representation
${\mathbb D}_r$ ($1\le r<t$) of the
quiver \eqref{x1.5} to a representation
${\mathbb Q}_r$  of a new quiver as
follows. We first replace each vertex
$i\in\{1,2,\dots,r+1\}$ of ${\mathbb
D}_r$ by the vertices
$i_1,\dots,i_{d_i}$, where
\[
d_i=\dim_i{\mathbb D}_r= \dim_i{\mathbb
A}.
\]
Then we replace the arrow $(r+1)\
\frac{A_{r+1}'}{\qquad}\ (r+2)$ by the
arrows\\[10mm]
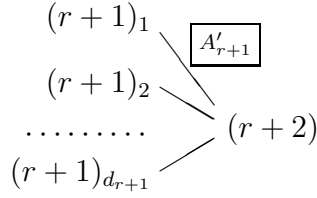
\[
\unitlength 0.4mm \linethickness{0.4pt}
\begin{picture}(25.00,25.00)(0,1)
\put(13.00,38.00){\makebox(0,0)[lc]
{\fbox{$\scriptstyle{ A_{r+1}'}$}}}
\put(0.00,25.00){\makebox(0,0)[rc]
{$(r+1)_2$}}
\put(-21.00,8.00){\makebox(0,0)
{$\cdots\cdots\cdots$}}
\put(0.00,-5.00){\makebox(0,0)[rc]
{$(r+1)_{d_{r+1}}$}}
\put(25.00,10.00){\makebox(0,0)[lc]{$(r+2)
$}}
\put(3.00,24.00){\line(5,-3){18.00}}
\put(3.00,-4.00){\line(5,3){18.00}}
\put(3.00,42.00){\line(3,-4){18.00}}
\put(0.00,47.00){\makebox(0,0)[rc]
{$(r+1)_1$}}
\end{picture}\\*[15pt]
\]
of the same direction.

At last, we replace each arrow $i
\stackrel{D_i}{\lllin}(i+1)$ with $i\le
r$ by arrows that are in one-to-one
correspondence with the units of the
matrix ${D}_i$: every unit at the place
$(p,q)$ in ${D}_i$ determines the arrow
\[
i_q\stackrel{I_{1}}{\longrightarrow}
(i+1)_p\qquad\text{if}\qquad
\alpha_i:i\longrightarrow (i+1),
\]
or the arrow
\[
i_p\stackrel{I_{1}}{\longleftarrow}
(i+1)_q\qquad\text{if}\qquad
\alpha_i:i\longleftarrow (i+1).\]
(These arrows represent the action on
the basic vectors of the linear
operator
\[
\ {\mathbb C} i_1\oplus\dots\oplus
{\mathbb C} i_{d_i}\lllin\: {\mathbb C}
(i+1)_1\oplus\dots\oplus {\mathbb C}
(i+1)_{d_{i+1}}
\]
directed as $\alpha_i:i\llin (i+1)$ and
given by the matrix $D_i$.) Since in
each row and in each column of $D_i$ at
most one entry is $1$ and the others
are $0$, two arrows $i_p\llin (i+1)_q$
and $i_{p'}\llin (i+1)_{q'}$ have no
common vertices (${\cal D}_i$ sends
each basic vector to a basic vector or
to $0$  and cannot send two basic
vectors to the same basic vector).
Denote the obtained  representation by
${\mathbb Q}_r$.

The quiver representation ${\mathbb
Q}_{t-1}$ is a union of nonintersecting
chains; each of them determines a
representation of the form ${\mathbb
L}_{ij}$. Hence, ${\mathbb D}_{t-1}$ is
a direct sum of these representations.
By \eqref{x3.y}, ${\mathbb A}\simeq
{\mathbb D}_{t-1}$, so ${\mathbb
D}_{t-1}$ is the canonical form of
${\mathbb A}$, and we need to prove
${\mathbb L}({\mathbb A})={\mathbb
D}_{t-1}$ (see \eqref{x2.1g}).

It suffices to show that
\[
{\mathbb
Q}_{t-1}={\cal P}_{t-1}\cup {\mathbb
M}_{t-1}
\]
(see the set of indices in
\eqref{x2.1g}). The equality ${\mathbb
Q}_{1}={\cal P}_{1}\cup {\mathbb
M}_{1}$ holds since the matrix $D_1$ is
obtained from $B_1$ by replacement of
$H$ with $I_k$ (see \eqref{x2.1'} and
\eqref{x2.ee}). Reasoning by induction,
we assume that ${\mathbb Q}_{r-1}={\cal
P}_{r-1}\cup {\mathbb M}_{r-1}$. Then
${\mathbb Q}_{r}={\cal P}_{r}\cup
{\mathbb M}_{r}$ by the construction of
${\cal P}_{r}$ and ${\mathbb M}_{r}$ in
step $r$ of the algorithm for chains
and since $D_r$ is obtained from $B_r$
by replacement of all blocks $H_{i}$ by
$I_{l_i}$ and all blocks $*$ by $0$.
This proves Proposition \ref{t2.a}.

\begin{example}
Suppose we apply the algorithm to a
matrix representation
\begin{equation*}\label{x1.4a}
{\mathbb A} :\qquad 1\
\stackrel{A_1}{\longrightarrow}\ 2\
\stackrel{A_2}{\longrightarrow}\ 3\
\stackrel{A_3}{\longleftarrow}\ 4
\end{equation*}
of dimension $(4,5,4,5)$ and obtain
\begin{equation*}\label{xzz}
B_1=\begin{bmatrix}
  0_{31} & H_{1} \\
  0_{21} & 0_{23}
\end{bmatrix},\
B_2=\left[\begin{tabular}{cc|cc}
   $0_{21}$&$H_{2}$ & $*$& $*$ \\
  $0_{11}$&$0_{12}$ &
  $0_{11}$&$H_{3}$\\
  $0_{11}$&$0_{12}$ &
  $0_{11}$&$0_{11}$
\end{tabular}
\right],\
B_3=\left[\begin{tabular}{ccc}
   $0_{22}$&$H_{4}$&$*$ \\  \hline
  $0_{12}$&$0_{12}$&$*$\\  \hline
  $0_{12}$&$0_{12}$&$H_{5}$
\end{tabular}
\right],
\end{equation*}
where $H_1,\dots,H_5$ are nonsingular
$3\times 3$, $2\times 2$, $1\times 1$,
$2\times 2$, and $1\times 1$ matrices.
Then
\[
\begin{CD}
{\mathbb D}_3:\qquad
1@>{\begin{bmatrix}
0&1&0&0\\0&0&1&0\\0&0&0&1\\0&0&0&0\\0&0&0&0
\end{bmatrix}}>>2
@>{\begin{bmatrix}
0&1&0&0&0\\0&0&1&0&0\\0&0&0&0&1\\0&0&0&0&0
\end{bmatrix}}>>3
@<{\begin{bmatrix}
0&0&1&0&0\\0&0&0&1&0\\0&0&0&0&0\\0&0&0&0&1
\end{bmatrix}}<<4
\end{CD}
\]
and\\[-16mm]
\begin{equation}\label{yyy}
\unitlength 1mm \linethickness{0.4pt}
\begin{picture}(110.00,40.00)(0,19)
\put(0.00,20.00){\makebox(0,0)[cc]{${\mathbb
Q}_3:$}}
\put(20.00,40.00){\makebox(0,0)[cc]{$1_1$}}
\put(20.00,30.00){\makebox(0,0)[cc]{$1_2$}}
\put(20.00,20.00){\makebox(0,0)[cc]{$1_3$}}
\put(20.00,10.00){\makebox(0,0)[cc]{$1_4$}}
\put(50.00,40.00){\makebox(0,0)[cc]{$2_1$}}
\put(50.00,30.00){\makebox(0,0)[cc]{$2_2$}}
\put(50.00,20.00){\makebox(0,0)[cc]{$2_3$}}
\put(50.00,10.00){\makebox(0,0)[cc]{$2_4$}}
\put(50.00,0.00){\makebox(0,0)[cc]{$2_5$}}
\put(80.00,40.00){\makebox(0,0)[cc]{$3_1$}}
\put(80.00,30.00){\makebox(0,0)[cc]{$3_2$}}
\put(80.00,20.00){\makebox(0,0)[cc]{$3_3$}}
\put(80.00,10.00){\makebox(0,0)[cc]{$3_4$}}
\put(110.00,40.00){\makebox(0,0)[cc]{$4_1$}}
\put(110.00,30.00){\makebox(0,0)[cc]{$4_2$}}
\put(110.00,20.00){\makebox(0,0)[cc]{$4_3$}}
\put(110.00,10.00){\makebox(0,0)[cc]{$4_4$}}
\put(110.00,0.00){\makebox(0,0)[cc]{$4_5$}}
\put(107.00,1.00){\vector(-3,1){24.00}}
\put(107.00,12.00){\vector(-3,2){24.00}}
\put(107.00,22.00){\vector(-3,2){24.00}}
\put(53.00,2.00){\vector(3,2){24.00}}
\put(53.00,21.00){\vector(3,1){24.00}}
\put(53.00,31.00){\vector(3,1){24.00}}
\put(23.00,11.00){\vector(3,1){24.00}}
\put(23.00,21.00){\vector(3,1){24.00}}
\put(23.00,31.00){\vector(3,1){24.00}}
\end{picture}\\*[25mm]
\end{equation}
\bigskip
We have the canonical form of ${\mathbb
A}$:
\[
 {\mathbb A}\simeq
 {\mathbb L}_{11}\oplus
 {\mathbb L}_{12}\oplus
 {\mathbb L}_{14}\oplus
 {\mathbb L}_{14}\oplus
 {\mathbb L}_{22}\oplus
 {\mathbb L}_{23}\oplus
 {\mathbb L}_{34}\oplus
 {\mathbb L}_{44}\oplus
 {\mathbb L}_{44}
\]
\end{example}

Note that the block-triangular form of
$S_{r}$ (see \eqref{x3.x}) follows from
the disposition of the chains
\[
p_i
\stackrel{I_{1}}{\lllin}
  (p_i+1)\stackrel{I_{1}}{\lllin}
  \cdots \stackrel{I_{1}}{\lllin} r,
\qquad i=1,\dots,r+1,
\]
in the quiver representation ${\mathbb
M}_{r-1}$ (see \eqref{x2.1c}): they
represent the linear mappings and we
may add these chains from the top down
by changing bases in vector spaces;
this is clear for the quiver
representation \eqref{yyy}.

\section{Proof of Proposition \ref{l2}}
 \label{subs}

The representation ${\mathbb A}^{(r)}$
(see \eqref{1ww}) is a representation
of the quiver, which we will denote by
${\cal Q}^{(r)}$. For every
representation
\\[1em]
\\[1.5em]
\begin{equation}\label{m1}
\unitlength 1.1mm \linethickness{0.4pt}
\begin{picture}(90.33,24.33)(1.5,20)
\put(-3.50,45.00){\makebox(0,0)[rc]
{${\mathbb B}$:}}
\put(0.00,25.00){\makebox(0,0)[rc]{$(t+1)'$}}
\put(1.00,25.00){\line(1,0){6.00}}
\put(13.50,25.00){\makebox(0,0)[cc]{$(t+2)'$}}
\put(19.33,25.00){\line(1,0){7.00}}
\multiput(28.5,23.9)(1.96,0){26}{$\cdot$}
\put(80.33,25.00){\line(1,0){8.00}}
\put(90.33,25.00){\makebox(0,0)[lc]{$(2t)'$}}
\bezier{348}(88.67,-0.67)(45.00,-10.00)(1.00,-0.67)
\put(0.00,10.00){\makebox(0,0)[rc]{$(kt+1)'$}}
\put(1.00,10.00){\line(1,0){7.00}}
\put(12.00,9.7){\makebox(0,0)[cc]{$\cdots$}}
\put(15.33,10.00){\line(1,0){8.00}}
\put(19.33,12.70){\makebox(0,0)[cc]
{$\scriptstyle{B_{(r-1)'}}$}}
\put(25.67,10.00){\makebox(0,0)[cc]{$r'$}}
\multiput(1.33,13)(2.01,0.23){44}{$\cdot$}
\put(0.00,0.00){\makebox(0,0)[rc]{$1$}}
\put(1.00,0.00){\line(1,0){7.00}}
\put(12.04,-0.250){\makebox(0,0)[cc]{$\cdots$}}
\put(15.33,0.00){\line(1,0){8.00}}
\put(19.33,3){\makebox(0,0)[cc]
{$\scriptstyle{B_{[r-1]}}$}}
\put(33.33,3.0){\makebox(0,0)[cc]
{$\scriptstyle{B_{[r]}}$}}
\put(33.33,13){\makebox(0,0)[cc]
{$\scriptstyle{B_{r'}}$}}
\put(25.67,0.00){\makebox(0,0)[cc]{$[r]$}}
\put(28.00,0.00){\line(1,0){9.00}}
\put(43.00,0.00){\makebox(0,0)[cc]{$[r+1]$}}
\put(48.00,0.00){\line(1,0){9.00}}
\put(63.00,0.00){\makebox(0,0)[cc]{$[r+2]$}}
\put(68.00,0.00){\line(1,0){9.00}}
\put(73.00,3.0){\makebox(0,0)[cc]
{$\scriptstyle{B_{[r+2]}}$}}
\put(80.7,-0.25){\makebox(0,0)[cc]{$\cdots$}}
\put(84,0.00){\line(1,0){4.30}}
\put(90.33,0.00){\makebox(0,0)[lc]{$t$}}
\put(1.00,26.33){\line(6,1){87.00}}
\put(58.5,8.33){\makebox(0,0)[cc]
{\fbox{$\scriptstyle{B}$}}}
\put(28.00,10.00){\line(1,0){9.00}}
\put(43.00,10.00){\makebox(0,0)[cc]{$(r+1)'$}}
\put(48.00,8.33){\line(3,-2){9}}
\put(92,43.8){\makebox(0,0)[rc]
{$(l+1)'
\stackrel{B_{(l+1)'}}{\lllin}(l+2)'
\stackrel{B_{(l+2)'}}{\lllin} \cdots\
  \frac{}{\qquad}\ t'$}}
\end{picture}\\*[33mm]
\end{equation}
of this quiver, we define the
representation
\begin{equation*}
\unitlength 1.00mm
\linethickness{0.4pt}
\begin{picture}(85.33,5.00)
(10,0)
\put(25.33,0.00){\makebox(0,0)[cc]{$1$}}
\put(65.33,0.00){\makebox(0,0)[cc]{$\cdots$}}
\put(85.33,0.00){\makebox(0,0)[cc]{$t$}}
\put(45.33,0.00){\makebox(0,0)[cc]{$2$}}
\put(30.33,0.00){\line(1,0){10.00}}
\put(50.33,0.00){\line(1,0){10.00}}
\put(70.33,0.00){\line(1,0){10.00}}
\put(35.33,5.00){\makebox(0,0)[ct]{${D}_1$}}
\put(55.33,5.00){\makebox(0,0)[ct]{${D}_2$}}
\put(75.33,5.00){\makebox(0,0)[ct]{${D}_{t-1}$}}
\put(0.33,0.00){\makebox(0,0)[cc]
{${\mathbb{F}(\mathbb{B}})$:}}
\put(55.33,-7.33){\makebox(0,0)[cc]{$D_t$}}
\put(29.83,-1.67){\line(-4,1){0.2}}
\bezier{208}(81.00,-1.67)(55.33,-7.33)(29.83,-1.67)
\end{picture}\\*[20pt]
\end{equation*}
of the cycle $\cal C$ by ``gluing down
of the shave'' (see the beginning of
Section \ref{s1.3}):
\begin{equation*}\label{1p}
D_i=\Bigl(\bigoplus_{\substack{[j]=i\\
   l\le j\le r}}{ B_{j'}}\Bigr)\oplus
  \begin{cases}
   B_i & \text{if $i\ne [r+1]$},
\\
   B & \text{if $i=[r+1]$},
  \end{cases}
\end{equation*}
where $B_{l'}$ is defined by
\eqref{jjj} (compare with \eqref{eee}).
The mapping $\mathbb F$ is analogous to
the ``push-down functor'' \eqref{2.3w}.
Moreover, for the representation
$\mathbb A^{(n)}$, obtained in the last
step of the algorithm for cycles, we
have
\begin{equation}\label{1o}
\mathbb{F}(\mathbb{A}^{(n)})=
\mathbb{P}({\mathbb A}')\oplus
\widetilde{\mathbb{A}},
\end{equation}
where ${\mathbb A}'$ and
$\widetilde{\mathbb{A}}$ are the
representations \eqref{1k} and
\eqref{1h}.

By \eqref{2.1.5aq}, the matrix $B$ in
\eqref{m1} has the form
\begin{equation*}\label{1w}
 B=\begin{cases}
 [\,B_{(r+1)'}\,|\,B_{[r+1]}\,]& \text{if
 $\alpha_{[r+1]}$ is oriented
clockwise,} \\[0.1em]
 \left[\begin{tabular}{c}
    $B_{(r+1)'}$\\ \hline $B_{[r+1]}$
\end{tabular}\right] &
\text{if $\alpha_{[r+1]}$ is oriented
counterclockwise}.
  \end{cases}
\end{equation*}
By {\it triangular transformations}
with a representation $\mathbb B$ of
the form \eqref{m1}, we mean the
following transformations:
\begin{itemize}
\item[(i)]
additions of linear combinations of
columns of $B_{[r+1]}$ to columns of
$B_{(r+1)'}$ if $\alpha_{[r+1]}$ is
oriented clockwise,

\item[(ii)]
additions of linear combinations of
rows of $B_{(r+1)'}$ to rows of
$B_{[r+1]}$ if $\alpha_{[r+1]}$ is
oriented counterclockwise.
\end{itemize}
We say that $\mathbb B$ is a {\it
triangular representation} if
\begin{equation*}\label{1t}
{\mathbb F}({\mathbb B})\simeq {\mathbb
F}({\mathbb B}^{\vartriangle})
\end{equation*}
for every representation ${\mathbb
B}^{\vartriangle}$ obtained from
${\mathbb B}$ by triangular
transformations.

\begin{lemma}\label{l.2e}
Suppose $\mathbb{D}$ is obtained from a
triangular representation $\mathbb B$
of ${\cal Q}^{(r)}$ by transformations
at the vertex $[r+2]$. Then
$\mathbb{D}$ is triangular too.
\end{lemma}

\begin{proof}
Let
\begin{equation*}\label{2.0'}
{\cal
S}=(I,\dots,I,S_{[r+2]},I,\dots,I):
 {\mathbb B}\is {\mathbb D}
\end{equation*}
(see \eqref{1.3}). We must prove that
${\mathbb F}(\mathbb D)\simeq{\mathbb
F}({\mathbb D}^{\vartriangle})$ for
every ${\mathbb D}^{\vartriangle}$
obtained from ${\mathbb D}$ by
triangular transformations. Denote by
${\mathbb B}^{\vartriangle}$ the matrix
representation obtained from ${\mathbb
B}$ by the same triangular
transformations. By \eqref{1.2aa} and
the definition of triangular
transformations, there is a block
matrix
\[
R=\begin{bmatrix}
 I&0\\ *&I
\end{bmatrix}
\]
such that
\begin{align*}
[D^{\vartriangle}_{(r+1)'}\,|\,
D^{\vartriangle}_{[r+1]}]
&=S_{[r+2]}[B_{(r+1)'}\,|\, B_{[r+1]}]
R\quad
 \text{if $\alpha_{[r+1]}$ is oriented
clockwise,}\\[3mm]
\left[\begin{tabular}{c}
    $D^{\vartriangle}_{(r+1)'}$\\ \hline
    $D^{\vartriangle}_{[r+1]}$
\end{tabular}\right]
&=R \left[\begin{tabular}{c}
    $B_{(r+1)'}$\\ \hline
    $B_{[r+1]}$
\end{tabular}\right]
S_{[r+2]}^{-1}\quad
  \text{if $\alpha_{[r+1]}$ is oriented
counterclockwise,}
\\[3mm]
D^{\vartriangle}_{[r+2]}&=
\begin{cases} B_{[r+2]} S_{[r+2]}^{-1}
  &\text{if $\alpha_{[r+2]}$ is oriented
clockwise,}\\ S_{[r+2]} B_{[r+2]}
  &\text{if $\alpha_{[r+2]}$ is oriented
counterclockwise}.
\end{cases}
\end{align*}
These equalities imply \[{\cal
S}=(I,\dots,I,S_{[r+2]},I,\dots,I):
{\mathbb B}^{\vartriangle} \is {\mathbb
D}^{\vartriangle}\] and
\[
{\mathbb F}({\mathbb D})\simeq {\mathbb
F}({\mathbb B})\simeq {\mathbb
F}({\mathbb B}^{\vartriangle}) \simeq
{\mathbb F}({\mathbb
D}^{\vartriangle}).
\]
\end{proof}

\begin{lemma}\label{l.2f}
Each representation ${\mathbb A}^{(r)}$
$($obtained in step $r$ of the
algorithm for cycles$)$ is triangular
and ${\mathbb F}({\mathbb
A}^{(r)})\simeq {\mathbb A}$.
\end{lemma}

\begin{proof}
The lemma is obvious if $l=t+1$ (see
\eqref{dd}). Suppose $l\le t$. The
statements hold for ${\mathbb
A}^{(1)},\dots,{\mathbb A}^{(l)}$.
Reasoning by induction, we assume that
they hold for ${\mathbb A}^{(r-1)}$
with $r-1\ge l$ and prove them for
${\mathbb A}^{(r)}$.

First we apply the unitary
transformations at the vertex $[r+1]$
from step $r$ of the algorithm for
cycles to the representation ${\mathbb
A}^{(r-1)}$ of the quiver ${\cal
Q}^{(r-1)}$: we reduce the matrix
$A^{(r-1)}$ to a block-triangular form
by transformations \eqref{2.3''a} or
\eqref{2.3'''} (depending on the
orientation of $\alpha_{[r]}$), and the
matrix $A_{[r+1]}^{(r-1)}$ to
$A^{(r)}$. Denote the obtained
representation by ${\mathbb
A}^{(r-2/3)}$.

Then we make zero the block $*$ of
\eqref{2.3''a} or \eqref{2.3'''} by
triangular transformations and obtain
the following representation ${\mathbb
A}^{(r-1/3)}$ of the quiver ${\cal
Q}^{(r-1)}$:
\\[1.5em]
\begin{equation*}
\unitlength 1.1mm \linethickness{0.4pt}
\begin{picture}(90.33,22.33)(-10,22)
\put(-10.00,45.00){\makebox(0,0)[rc]
{${\mathbb A}^{(r-1/3)}$:}}
\put(0.00,25.00){\makebox(0,0)[rc]{$(t+1)'$}}
\put(1.00,25.00){\line(1,0){5.00}}
\put(13.50,25.00){\makebox(0,0)[cc]{$(t+2)'$}}
\put(20.33,25.00){\line(1,0){6.0}}
\multiput(28.,24.0)(2,0){26}{$\cdot$}
\put(80.33,25.00){\line(1,0){8.00}}
\put(90.33,25.00){\makebox(0,0)[lc]{$(2t)'$}}
\bezier{348}(88.67,-0.67)(45.00,-10.00)(1.00,-0.67)\put(0.00,10.00){\makebox(0,0)[rc]{$(kt+1)'$}}
\put(1.00,10.00){\line(1,0){7.00}}
\put(12.00,9.70){\makebox(0,0)[cc]{$\cdots$}}
\put(15.33,10.00){\line(1,0){8.00}}
\put(19.33,13.3){\makebox(0,0)[cc]
{$\scriptstyle{A_{(r-1)'}^{(r-1)}}$}}
\put(25.67,10.00){\makebox(0,0)[cc]{$r'$}}
\multiput(1.33,11.67)(2,0.25){6}{$\cdot$}
\multiput(25,14.3)(2,0.25){32}{$\cdot$}
\put(0.00,0.00){\makebox(0,0)[rc]{$1$}}
\put(1.00,0.00){\line(1,0){7.00}}
\put(12.00,-0.30){\makebox(0,0)[cc]{$\cdots$}}
\put(15.33,0.00){\line(1,0){7.0}}
\put(18.7,3.30){\makebox(0,0)[cc]
{$\scriptstyle{A_{[r-1]}^{(r-1)}}$}}
\put(25.67,0.00){\makebox(0,0)[cc]{$[r]$}}
\put(28.00,0.00){\line(1,0){8.50}}
\put(43.00,0.00){\makebox(0,0)[cc]{$[r+1]$}}
\put(49.00,0.00){\line(1,0){7.20}}
\put(53.00,2.50){\makebox(0,0)[cc]
{$\scriptstyle{A^{(r)}}$}}
\put(63.00,0.00){\makebox(0,0)[cc]{$[r+2]$}}
\put(69.00,0.00){\line(1,0){7.00}}
\put(73.00,3.30){\makebox(0,0)[cc]
{$\scriptstyle{A_{[r+2]}^{(r-1)}}$}}
\put(80.7,-0.20){\makebox(0,0)[cc]{$\cdots$}}
\put(84,0.00){\line(1,0){5.00}}
\put(90.33,0.00){\makebox(0,0)[lc]{$t$}}
\put(27.33,9.33){\line(4,-3){9.00}}
\put(92,45){\makebox(0,0)[rc] {$(l+1)'
\stackrel{A^{(r-1)}_{(l+1)'}}{\lllin}
(l+2)'
\stackrel{A^{(r-1)}_{(l+2)'}}{\lllin}
\cdots
  \lllin\ t'$}}
\put(1.00,27){\line(6,1){87.3}}
\put(43.00,10){\makebox(0,0)[cc]
{\fbox{$\scriptstyle{A_{r'}^{(r)}\oplus
A_{[r]}^{(r)}}$}}}
\put(45.16,-8.5){\makebox(0,0)[cc]
{$\scriptstyle{A_{t}^{(r-1)}}$}}
\end{picture}\\*[9.5em]
\end{equation*}
By the induction hypothesis, ${\mathbb
A}\simeq{\mathbb F}({\mathbb
A}^{(r-1)})$ and ${\mathbb A}^{(r-1)}$
is triangular. By Lemma \ref{l.2e},
${\mathbb A}^{(r-2/3)}$ is triangular
too, and so
\[
{\mathbb F}({\mathbb
A}^{(r-2/3)})\simeq{\mathbb F}({\mathbb
A}^{(r-1/3)}).
\]
We have
\[
{\mathbb A}\simeq{\mathbb F}({\mathbb
A}^{(r-1)})\simeq{\mathbb F}({\mathbb
A}^{(r-2/3)})\simeq{\mathbb F}({\mathbb
A}^{(r-1/3)})={\mathbb F}({\mathbb
A}^{(r)}).
\]

Let ${\mathbb A}^{(r)\vartriangle}$ be
obtained from ${\mathbb A}^{(r)}$ by
triangular transformations. These
transformations reduce $A^{(r)}$ (see
\eqref{1ww}) to a new matrix
$A^{(r)\vartriangle}$ and do not change
the other matrices of ${\mathbb
A}^{(r)}$. Since
\[
A^{(r)}=A_{[r+1]}^{(r-1/3)},
\]
these transformations with
$A_{[r+1]}^{(r-1/3)}$ can be realized
by transformations at the vertex
$[r+1]$ of ${\mathbb A}^{(r-1/3)}$;
denote the obtained representation by
${\mathbb A}^{(r-1/3)\vartriangle}$, it
is triangular by Lemma \ref{l.2e}.
These transformations may spoil the
subdiagonal block $0$ of
\[
A^{(r-1/3)}={A_{r'}^{(r)}\oplus
A_{[r]}^{(r)}},
\]
but it is recovered
by triangular transformations and so
\[
{\mathbb F}({\mathbb
A}^{(r-1/3)\vartriangle}) \simeq
{\mathbb F}({\mathbb
A}^{(r)\vartriangle}).
\]
Since
\[
{\mathbb F}({\mathbb A}^{(r)})=
{\mathbb F}({\mathbb
A}^{(r-1/3)})\simeq {\mathbb
F}({\mathbb A}^{(r-1/3)\vartriangle})
\simeq {\mathbb F}({\mathbb
A}^{(r)\vartriangle}),
\]
the representation ${\mathbb
F}({\mathbb A}^{(r)})$ is triangular.
\end{proof}

\begin{lemma}\label{l.2g}
Let ${\mathbb A}^{(k)}$ be the
representation obtained from a
representation ${\mathbb A}$ in step
$k$ of the algorithm for cycles, and
let $k\ge l$ $($hence $l\le t$ by
\eqref{dd} and \eqref{mmm}$)$. Denote
\begin{equation*}\label{lll}
\widehat{A}^{(k)}_i=
  \begin{cases}
    {A}^{(k)}_i& \text{if}\ \
                i\ne [k+1], \\
    {A}^{(k)} & \text{if}\ \
                i= [k+1],
  \end{cases}
\end{equation*}
where $i=1,\dots,t$. Then
\begin{itemize}
  \item[\rm(i)]
The rows of $\widehat{A}^{(k)}_i$ are
linearly independent if $\alpha_i$ is
oriented clockwise and $i\le k$.
  \item[\rm(ii)]
The columns of $\widehat{A}^{(k)}_i$
are linearly independent if $\alpha_i$
is oriented counterclockwise and the
columns of ${A}_i$ are linearly
independent.
\end{itemize}
\end{lemma}

\begin{proof}
We will prove the lemma by induction on
$k$. Clearly, the statements (i) and
(ii) hold for $k=l$. Assume they hold
for $k=r-1\ge l$ and prove them for
$k=r$. We need to check (i) and (ii)
only for $i=[r]$ and $i=[r+1]$ since in
step $r$ of the algorithm we change
$\widehat{A}^{(r-1)}_{[r]}$ and
$\widehat{A}^{(r-1)}_{[r+1]}$.

By \eqref{2.3''a}, the matrix $\widehat
A_{[r]}^{(r)}=A_{[r]}^{(r)}$ has
linearly independent rows if
$\alpha_{[r]}$ is oriented clockwise.
By \eqref{2.3'''}, this matrix has
linearly independent columns if both
$\alpha_{[r]}$ is oriented
counterclockwise and $\widehat
A_{[r]}^{(r-1)}=A^{(r-1)}$ has linearly
independent columns. Hence, (i) and
(ii) hold for $i=[r]$.

The statements (i) and (ii) hold for
$i=[r+1]$ by the induction hypothesis
and since $\widehat
A_{[r+1]}^{(r)}=A^{(r)}$ is obtained
from $A_{[r+1]}^{(r-1)}$ by elementary
transformations with its columns or
rows.
\end{proof}

\begin{proof}[Proof of Proposition \ref{l2}]
The statement (c) of Proposition
\ref{l2} follows from \eqref{1o} and
Lemma \ref{l.2f}, so we will prove (a)
and (b).

If $l=t+1$ (see \eqref{dd}), then
$\widetilde{\mathbb{A}}=\mathbb{A}$
satisfies (a) and (b).

Suppose $l\le t$. Then
$\widetilde{\mathbb{A}}$ is the
restriction of the representation
${\mathbb{A}}^{(n)}$ (obtained in the
last step of the algorithm) to the
cycle $\cal{C}$ and so $
\widetilde{A}_i=A^{(n)}_i$
($i=1,2,\dots,t$).

Since
\[
\widehat
A^{(n)}_i=A^{(n)}_i=\widetilde{A}_i
%\qquad\text{if\quad $i\ne [n+1]$},
\]
if $i\ne [n+1]$,
\[
\widehat A^{(n)}_{[n+1]}
 =A^{(n)}=
 \left[\begin{tabular}{c|c}
$0$&
 $A^{(n)}_{[n+1]}$
 \end{tabular}\right]=
 \left[\begin{tabular}{c|c}
$0$&
 $\widetilde{A}_{[n+1]}$
 \end{tabular}\right]
\]
if $\alpha_{[n+1]}$ is oriented
clockwise (see \eqref{vvv} and
\eqref{wwv}), and
\[
\widehat A^{(n)}_{[n+1]}
 =A^{(n)}=
 \left[\begin{tabular}{c}
   $0$\\ \hline
    \raisebox{-3pt}{$A^{(n)}_{[n+1]}$}
\end{tabular}\right]=
\left[\begin{tabular}{c}
   $0$\\ \hline
    \raisebox{-3pt}{$\widetilde{A}_{[n+1]}$}
\end{tabular}\right]
\]
if $\alpha_{[n+1]}$ is oriented
counterclockwise, the statements (i)
and (ii) follow from Lemma \ref{l.2g},
in which $k=n\ge t$.
\end{proof}
\medskip

{\it The author wishes to express his
gratitude to Professor Roger Horn for
the hospitality and stimulating
discussions.}

\end{document}